\documentclass[final,leqno]{siamltex704}
\usepackage{hyperref}
\usepackage{amsmath}
\usepackage{amsfonts,amssymb}
\usepackage{array}
\usepackage{lineno}
\usepackage{color}
\newtheorem{remark}{Remark}[section]
\newcommand{\bn}{\textbf{n}}

\def\E{{\mathcal E}}

\def\3bar{{|\hspace{-.02in}|\hspace{-.02in}|}}
\def\O{\Omega}
\def\pa{\partial}
\newtheorem{gWG}{Generalized weak Galerkin Scheme}

\begin{document}

\setlength{\parindent}{0.25in} \setlength{\parskip}{0.08in}

\title{Generalized Weak Galerkin Finite Element Methods for Biharmonic Equations}

\author{
Dan Li\thanks{School of Mathematical Sciences,  Nanjing Normal University, Nanjing 210023, China (danlimath@163.com).}
\and
Chunmei Wang \thanks{Department of Mathematics, University of Florida, Gainesville, FL 32611 (chunmei.wang@ufl.edu).  The research of Chunmei Wang was partially supported by National Science Foundation Grants DMS-2136380 and DMS-2206332.}
\and
Junping Wang\thanks{Division of Mathematical Sciences, National Science Foundation, Alexandria, VA 22314 (jwang@nsf.gov). The research of Junping Wang was supported by the NSF IR/D program, while working at National Science Foundation. However, any opinion, finding, and conclusions or recommendations expressed in this material are those of the author and do not necessarily reflect the views of the National Science Foundation.}}

\maketitle

\begin{abstract}
The generalized weak Galerkin ({\rm g}WG) finite element method is proposed and analyzed for the biharmonic equation. A new   generalized discrete weak second order partial derivative  is introduced in the {\rm g}WG scheme  to allow arbitrary combinations of piecewise polynomial functions defined in the interior and on the boundary of general polygonal or polyhedral elements. The error estimates are established for the numerical approximation in a discrete $H^2$ norm and a  $L^2$ norm.  The numerical results are reported to demonstrate the accuracy and flexibility of our proposed {\rm g}WG method for the biharmonic equation.
\end{abstract}

\begin{keywords}  {\rm g}WG, weak Galerkin, finite element methods,  generalized discrete weak second order partial derivative, biharmonic equation, polytopal partitions.
\end{keywords}

\begin{AMS}
Primary 65N30, 65N12, 65N15; Secondary 35B45, 35J50.
\end{AMS}

\section{Introduction}
This paper is concerned with the new development of the generalized weak Galerkin finite element method for the biharmonic equation. For simplicity, we consider the biharmonic equation that seeks an unknown function $u$ satisfying
\begin{equation}\label{model-problem}
\begin{split}
\Delta^2u&=f,\quad\mbox{in}~~\O\subset \mathbb{R}^d,\\
        u&= g_1, \quad \mbox{on}~~\pa\O,\\
\frac{\pa u}{\pa\textbf{n}}&= g_2, \quad\mbox{on}~~\pa\O,
\end{split}
\end{equation}
where $d=2,3$, $\O$ is an open bounded domain with Lipschitz continuous boundary $\pa\O$, the functions $f$,   $g_1$ and $g_2$ are  given on $\O$ or $\pa\O$ as appropriate, and $\textbf{n}$ is an unit outward normal direction to $\pa\O.$

A weak formulation for \eqref{model-problem} seeks  $u\in H^2(\O)$ satisfying $u|_{\pa\O}=g_1$ and $\frac{\pa u}{\pa\textbf{n}}|_{\pa\O}=g_2$ such that
\begin{equation}\label{variation-form}
\sum_{i,j=1}^d(\pa_{ij}^2u,\pa_{ij}^2v)=(f,v),~~~~\forall v\in H_0^2(\O),
\end{equation}
where $H_0^2(\O)=\{v\in H^2(\O):v|_{\pa\O}=0, \nabla v |_{\pa\O}=\textbf{0}\}$.  

The biharmonic equation has extensive applications in fluid dynamics. The conforming finite element methods have been developed for the biharmonic equation  by constructing finite element spaces as subspaces of $H^2(\Omega)$. The $H^2$-conforming methods essentially require $C^1$-continuity for the
underlying piecewise polynomials (known as finite element functions) on a prescribed finite element partition. The $C^1$-continuity imposes an enormous difficulty in the construction of the corresponding finite element functions in practical
computation. Due to the complexity in the construction of $C^1$-continuous elements, $H^2$-conforming finite element methods are rarely used in practice to solve the biharmonic equation.
As an alternative approach, the nonconforming and discontinuous Galerkin (DG) finite element methods have been employed to solve the biharmonic equation, such as the Morley element \cite{WWL-2022,LM19688,WX2007-3}, the $C^0$ interior penalty method \cite{BS2005}, the hp-version interior-penalty DG method \cite{MSB2007}, and the hybridizable DG method \cite{CDG2009}.  Recently, the weak Galerkin (WG) methods \cite{WWL-2022,mwy-biharmonic,Eff-MWY2017,WW_bihar-2014,WW_HWG-2015, RZZ2015},  the virtual element methods \cite{AMV2018} and the hybrid high-order methods \cite{DA2021} have been developed to solve the biharmonic equation.

Weak Galerkin finite element method is a newly-developed numerical technique for PDEs. In the WG method,  the differential operators in the variational formulation are approximated by a framework which mimics the theory of distributions for piecewise polynomials. The usual regularity of the approximating functions is compensated by carefully-designed stabilizers.  WG methods have been investigated for solving numerous model PDEs  \cite{li-wang,ww12,mwy,mwy-biharmonic,ww2,ww8,ww6,wy,ww4} and  have shown its great potential as a powerful numerical technique in scientific computing. The fundamental difference between the WG methods and other existing finite element methods is the use of weak derivatives and weak continuities in the design of numerical schemes based on conventional weak forms for the underlying PDE problems. Due to its great structural flexibility, WG methods are well suited to a wide class of PDEs by providing the needed stability and accuracy in approximations. A recent development of WG, named ``Primal-Dual Weak Galerkin (PDWG)'' has been proposed for problems for which the usual numerical methods are difficult to apply \cite{CWW2023,cwwdivcurl,cwcondif,cwwinterface,lwwhyper,lwspdwg,wcauchy2,wcauchy,wmpdwg,wwtrans,wwcauchy,wwnondiv,wzcondif}. The essential idea of PDWG is to interpret the numerical solutions as a constrained minimization of some functionals with constraints that mimic the weak formulation of the PDEs by using weak derivatives. The resulting Euler-Lagrange equation offers a symmetric scheme involving both the primal variable and the dual variable (Lagrange multiplier).  The PDWG methods have also been  extended to a more general $L^p$ theory   \cite{CWW2023, CWWsecond, CWWdiv,WWL2022}.

In this paper, we propose a novel generalized weak Galerkin method ({\rm g}WG) to solve the biharmonic equation. The {\textbf {novelty}} of the {\rm g}WG method is to propose:  1) arbitrary combinations of piecewise polynomials on the general polytopal partitions; and 2)  the  generalized discrete weak second order partial derivative. Therefore,  the {\rm g}WG method provides a high flexibility in solving a wide range of PDEs.

This paper is organized as follows. In Section \ref{Section:WeakHessian}, we define  the generalized discrete weak second order partial derivative. In Section \ref{Section:WG-scheme}, the {\rm g}WG scheme for the biharmonic model equation \eqref{model-problem} is proposed and the solution existence and uniqueness is presented. Section \ref{Section:EE} is devoted to a  derivation of an error equation for the proposed {\rm g}WG scheme.  In Section \ref{technique-estimate},   some technical results are presented. Sections \ref{error estimate}-\ref{errorl2}
are devoted to establishing  the error estimates for the numerical approximation in a discrete $H^2$-norm and a usual $L^2$-norm. Finally,  a series of numerical results are provided to demonstrate the accuracy and efficiency of the proposed {\rm g}WG method for the biharmonic equation  \eqref{model-problem} in Section \ref{Section:NE}.

This paper will follow some standard notations for the Sobolev spaces and norms. Let $D\subset\mathbb{R}^d (d=2,3)$ be any open bounded domain with Lipschitz continuous boundary. In the Sobolev space $H^s(D)$ for any integer $s\geq 0$, we denote the inner product, seminorm and norm by $(\cdot,\cdot)_{s, D}$, $|\cdot|_{s, D}$ and $\|\cdot\|_{s,D}$, respectively. When $s=0$,  the inner product, seminorm and norm are denoted by $(\cdot, \cdot)_{D}$, $|\cdot|_{D}$ and $\|\cdot\|_{D}$, respectively. When $D=\O$, the subscript $D$ is dropped in the corresponding inner product, seminorm and norm.  We use ``$\lesssim$'' to represent ``no greater than a generic positive constant independent of the meshsize or functions appearing in the inequalities''.

\section{Generalized discrete weak second order partial derivative}\label{Section:WeakHessian}

The goal of this section is to introduce the definition of the  generalized discrete weak second order partial derivative. To this end, let ${\cal T}_h$ be a polygonal or polyhedral partition of the domain $\O$ that satisfies the shape regular  assumption specified in \cite{ellip_MC2014}. Denote by $\E_h$ the set of all edges or flat faces in ${\cal T}_h$ and $\E_h^0=\E_h\setminus\pa\O$ the set of all interior edges or flat faces of ${\cal T}_h$. For each polytopal element $T\in{\cal T}_h$, denote by $h_T$ the diameter of $T$ and $h=\max_{T\in{\cal T}_h}h_T$ the meshsize of ${\cal T}_h$.
Denote by $P_r(T)$ the set of polynomials defined on $T$ with total degree no more than $r$.

Let $T\in{\cal T}_h$ be an element with boundary $\pa T$.
We introduce a weak function $v=\{v_0,v_b,\pmb{v_g}\}$ such that $v_0\in L^2(T)$, $v_b\in L^2(\partial T)$ and $\pmb{v_g}\in[L^2(\partial T)]^d$ ($d=2, 3$). The first and second components $v_0$ and $v_b$ represent  the  values  of the weak function $v$ in the interior and on the boundary of $T$, respectively. The third component $\pmb{v_g}=(v_{g1},\ldots,v_{gd})$ represents the gradient of $v$ on the boundary of $T$. It should be pointed out that $v_b$ and $\pmb{v_g}$ may not   necessarily be related to the traces of $v_0$ and $\nabla v_0$ on $\pa T$, respectively.
Denote by $V_{k,m,\ell}(T)$ the local weak finite element space on $T$; i.e.,
\begin{equation*}
V_{k,m,\ell}(T)=\{v=\{v_0,v_b,\pmb{v_g}\}:v_0\in P_k(T),v_b\in P_m(e),\pmb{v_g}\in[P_\ell(e)]^d,~e\subset\pa T\},
\end{equation*}
where $k\geq2$,  $m\geq0$ and $\ell\geq0$ are any given integers.

For each edge or face $e\in\E_h$, denote by $Q_b$ and $Q_{g}$  the usual $L^2$ projection operators onto $P_m(e)$ and $P_\ell(e)$ respectively. Denote by $\pmb{Q_g}=(Q_{g1},\ldots,Q_{gd})$ the usual $L^2$ projection operator onto $[P_\ell(e)]^d$.

\begin{definition}\label{def}
(Generalized discrete weak second order partial derivative)
 Let $n\geq0$ be any given integer. For $i, j=1,\cdots, d$, a generalized discrete weak second order partial  derivative for any weak function $v\in V_{k,m,\ell}(T)$, denoted by $\pa_{ij,g,T}^2v$, is given by 
\begin{equation}\label{weak Hessian-1}
 \pa_{ij,g,T}^2v=\pa_{ij}^2v_0+\delta_gv, 
\end{equation}
where $\delta_{g}v$ is defined as a linear functional in $P_n(T)$ such that
\begin{equation}\label{weak Hessian-2}
(\delta_gv,\varphi)_T=\langle(Q_bv_0-v_b)n_i,\pa_j\varphi\rangle_{\pa T}-\langle Q_{gi}(\nabla v_0)-v_{gi},\varphi n_j\rangle_{\pa T},
\end{equation}
for any $\varphi\in P_n(T)$. Here, $\bn=(n_1,\ldots,n_d)$ is an unit outward normal direction to $\pa T$.

\end{definition}


\section{Generalized weak Galerkin schemes}\label{Section:WG-scheme}

This section presents the {\rm g}WG method for the weak form \eqref{variation-form} of the biharmonic model equation \eqref{model-problem}.

A global weak finite element space $V_h$ is obtained by patching the local weak finite element space $V_{k,m,\ell}(T)$ over all the elements $T\in{\cal T}_h$ through a common value $v_b$ on the interior edges/faces $e\in\E_h^0$; i.e.,
\begin{equation*}
V_h=\{v=\{v_0,v_b,\pmb{v_g}\}:v|_T\in V_{k,m,\ell}(T),~T\in{\cal T}_h\}.
\end{equation*}
Denote by $V_h^0$ a subspace of the global weak finite element space $V_h$ given by
$$
V_h^0=\{v:v\in V_h,v_b=0,~\pmb{v_g}=\pmb{0},~\text{on}~e\subset\pa T\cap\pa\O\}.
$$

For simplicity, we denote by $\pa^2_{ij,g}v$ the  generalized discrete weak second order partial derivative $\pa_{ij,g,T}^2v$ computed by  Definition \ref{def}; i.e.,
\begin{equation*}
 (\pa^2_{ij,g} v)|_T =\pa_{ij,g,T}^2(v|_T),\qquad v\in V_h. 
\end{equation*}

For any $w,v\in V_h$, we introduce the following two bilinear forms:
\begin{equation*}\label{stabilizer}
\begin{split}
(\pa^2_gw,\pa^2_gv)_{{\cal T}_h}=&\sum_{T\in{\cal T}_h}\sum_{i,j=1}^d(\pa_{ij,g}^2w,\pa_{ij,g}^2v)_T,\\
s(w,v)=&\sum_{T\in{\cal T}_h}\rho_1h_T^{\gamma_1}\langle Q_b w_0-w_b, Q_b v_0-v_b\rangle_{\pa T}\\
        &+\sum_{T\in{\cal T}_h}\rho_2h_T^{\gamma_2}\langle \pmb{Q_g}(\nabla w_0)-\pmb{w_g},\pmb{Q_g}(\nabla v_0)-\pmb{v_g}\rangle_{\pa T},
\end{split}
\end{equation*}
 where $\rho_i>0$ for $i=1,2$, and $\gamma_i\in \mathbb R$ for $i=1,2$.

 The {\rm g}WG scheme for the biharmonic equation \eqref{model-problem} based on the variational formulation \eqref{variation-form} is given as follows: 

\begin{gWG}
 Find $u_h=\{u_0,u_b,\pmb{u_g}\}\in V_h$ such that $u_b=Q_bg_1$, $\pmb{u_g}\cdot\bn=Q_gg_2$ and $\pmb{u_g}\cdot\pmb{\tau}=Q_g(\nabla g_1\cdot\pmb{\tau})$ on $\pa\O$ satisfying 
\begin{equation}\label{WG-scheme}
(\pa^2_gu_h,\pa^2_gv)_{{\cal T}_h}+s(u_h,v)=(f,v_0),\qquad\forall v\in V_h^0,
\end{equation}
\end{gWG}
where $\pmb{\tau}$ is an unit tangential vector to the edge or face $e\subset\pa\O.$

For any $v\in V_h$, let us introduce a seminorm given by
\begin{equation}\label{triple-bar}
\3barv\3bar^2=(\pa^2_gv,\pa^2_gv)_{{\cal T}_h}+s(v,v).
\end{equation}

\begin{lemma}\label{tri-bar-prove}
For any $v\in V_h$, the seminorm $\3barv\3bar$ defined by \eqref{triple-bar} is a norm in the linear space $V_h^0$.
\end{lemma}
\begin{proof}
It suffices to verify the positivity property for $\3bar\cdot\3bar$. To this end, for any $v\in V_h^0$, it follows from $\3barv\3bar =0$   that $\pa_{ij,g}^2v=0$ on each element $T\in{\cal T}_h$ for $i,j=1,\ldots,d$ and $s(v,v)=0$. This leads to $Q_bv_0=v_b$ and $\pmb{Q_g}(\nabla v_0)= \pmb{v_g}$ on each $\pa T$, which, together with  $\pa^2_{ij,g}v=0$ on each $T$ and \eqref{weak Hessian-1}-\eqref{weak Hessian-2}, gives $\delta_g v=0$ and further $\pa^2_{ij}v_0=0$ for $i,j=1,\ldots,d$. Therefore, we obtain  $\nabla v_0=const$ on each $T$.  This, from the fact $\pmb{Q_g}(\nabla v_0)= \pmb{v_g}$ on each $\pa T$, gives $\nabla v_0\in C^0(\O)$. Further, using the boundary condition $\pmb{v_g}=\pmb{0}$ on $\pa\O$, we obtain $\nabla v_0=0$ in $\O$ and $\pmb{v_g}=0$ on each $\pa T$. Therefore, $
v_0=const$ on each $T$. Since $Q_bv_0=v_b$ on each $\pa T$, we have $v_0=v_b$ on each $\pa T$ and further $v_0\in C^0(\O)$. This, together with the boundary condition $v_b=0$ on $\pa\O$, gives $v_0=0$ on each $T$ and $v_b=0$ on each $\pa T$. This completes the proof.
\end{proof}

\begin{lemma}\label{uniqueness proper}
The generalized weak Galerkin scheme \eqref{WG-scheme} has one and only one  solution.
\end{lemma}
\begin{proof}
It suffices to prove that the {\rm g}WG scheme \eqref{WG-scheme} has a unique solution. To this end, assume that $u_h^{(1)}$ and $u_h^{(2)}$ are two different solutions arising from the {\rm g}WG scheme \eqref{WG-scheme}. Letting $v=u_h^{(1)}-u_h^{(2)}\in V_h^0$ in \eqref{WG-scheme}, there holds
$$(\pa^2_g(u_h^{(1)}-u_h^{(2)}),\pa^2_g(u_h^{(1)}-u_h^{(2)}))_{{\cal T}_h}+s(u_h^{(1)}-u_h^{(2)},u_h^{(1)}-u_h^{(2)})=0,$$
which, together with Lemma \ref{tri-bar-prove}, leads to $u_h^{(1)}=u_h^{(2)}$. This completes the proof.
\end{proof}


\section{Error equations}\label{Section:EE}

The objective of this section is to derive an error equation for the {\rm g}WG scheme \eqref{WG-scheme}. To this end, on each element $T\in{\cal T}_h$, denote by $Q_0$ the usual $L^2$ projection operator from $L^2(T)$ onto $P_k(T)$. We further define a projection operator $Q_h$ in the sense  that
$$
Q_h\phi=\{Q_0\phi,Q_b\phi,\pmb{Q_g}(\nabla\phi)\},\qquad\forall \phi\in H^2(\O).
$$
Denote by $\mathbb{Q}_s$ the usual $L^2$ projection operator onto $P_s(T)$ for $s=\min\{k,m,\ell,n\}$.

\begin{lemma}\label{pre-error-equation}
For any $\phi\in H^2(T)$ and $\varphi\in P_s(T)$, there holds
$$
(\pa_{ij,g}^{2}Q_h\phi,\varphi)_T=(Q_0\phi-\phi,\pa^2_{ji}\varphi)_T+(\pa^2_{ij}\phi,\varphi)_T.
$$
\end{lemma}
\begin{proof}
Using   \eqref{weak Hessian-1}-\eqref{weak Hessian-2}, the property of $L^2$ projection operators and the usual integration by parts, we have
\begin{equation*}
\begin{split}
 &(\pa_{ij,g}^2Q_h\phi,\varphi)_T\\
=&(\pa^2_{ij}Q_0\phi+\delta_gQ_h\phi,\varphi)_T\\
=&(\pa^2_{ij}Q_0\phi,\varphi)_T+\langle(Q_b(Q_0\phi)-Q_b\phi)n_i,\pa_j\varphi\rangle_{\pa T}
   -\langle Q_{gi}(\nabla Q_0\phi)-Q_{gi}(\nabla\phi),\varphi n_j\rangle_{\pa T}\\
 =&(\pa^2_{ij}Q_0\phi,\varphi)_T+\langle(Q_0\phi-\phi)n_i,\pa_j\varphi\rangle_{\pa T}-\langle\pa_iQ_0\phi-\pa_i\phi,\varphi n_j\rangle_{\pa T}\\
=&(\pa^2_{ij}Q_0\phi,\varphi)_T+(Q_0\phi-\phi,\pa^2_{ji}\varphi)_T-(\pa^2_{ij}(Q_0\phi-\phi),\varphi)_T \\
=&(Q_0\phi-\phi,\pa^2_{ji}\varphi)_T+(\pa^2_{ij}\phi,\varphi)_T.
\end{split}
\end{equation*}
This completes the proof.
\end{proof}

\begin{lemma}\label{error-equation}
Let $u$ and $u_h\in V_h$ be   the exact solution of the model equation \eqref{model-problem} and the   numerical solution arising from {\rm gWG} scheme \eqref{WG-scheme}, respectively. Let $e_h=Q_hu-u_h$ be the error function. Then, the error function $e_h$ satisfies the following error equation
\begin{eqnarray}\label{Error-equation}
(\pa^2_ge_h,\pa^2_gv)_{{\cal T}_h}+s(e_h,v)=\zeta_u(v),~~~~~\forall v\in V_h^0,
\end{eqnarray}
where the  term $\zeta_u(v)$ is given by
\begin{equation}\label{error equation-remainder}
\begin{split}
&\zeta_u(v)\\
=&s(Q_hu,v)+\sum_{T\in{\cal T}_h}\sum_{i,j=1}^d(Q_0u-u,\pa^2_{ji}(\mathbb{Q}_s\pa^2_{ij,g}v))_{T}+(\pa^2_{ij}v_0,(\mathbb{Q}_s-I)\pa^2_{ij}u)_T\\
 &+\langle(v_0-v_b)n_i,\pa_j(\mathbb{Q}_s\pa^2_{ij}u-\pa^2_{ij}u)\rangle_{\pa T}
    +\langle\pa_iv_0-v_{gi},(I-\mathbb{Q}_s)\pa^2_{ij}u\cdot n_j\rangle_{\pa T}\\
 &+(\pa^2_{ij,g}Q_hu,(I-\mathbb{Q}_s)\pa^2_{ij,g}v)_T.
\end{split}
\end{equation}
\end{lemma}

\begin{proof}
Testing the model equation \eqref{model-problem} against $v_0$ gives
\begin{equation}\label{error-equation-0}
\begin{split}
(f,v_0)
= &\sum_{T\in{\cal T}_h}(\Delta^2 u,v_0)_T\\
=&\sum_{T\in{\cal T}_h}\sum_{i,j=1}^d(\pa^2_{ij}u,\pa^2_{ij}v_0)_T
  -\langle\pa^2_{ij}u,\pa_iv_0\cdot n_j\rangle_{\pa T}+\langle \pa_j(\pa^2_{ij}u)\cdot n_i,v_0\rangle_{\pa T}\\
=&\sum_{T\in{\cal T}_h}\sum_{i,j=1}^d((I-\mathbb{Q}_s)\pa^2_{ij}u,\pa^2_{ij}v_0)_T+(\mathbb{Q}_s\pa^2_{ij}u,\pa^2_{ij}v_0)_T
 \\
  & -\langle\pa^2_{ij}u,(\pa_iv_0-v_{gi})\cdot n_j\rangle_{\pa T} +\langle \pa_j(\pa^2_{ij}u)\cdot n_i,v_0-v_b\rangle_{\pa T},
\end{split}
\end{equation}
where we also used
$$
\sum_{T\in{\cal T}_h}\sum_{i,j=1}^d\langle\pa^2_{ij}u,v_{gi}\cdot n_j\rangle_{\pa T}=0,
$$
 $$
\sum_{T\in{\cal T}_h}\sum_{i,j=1}^d\langle \pa_j(\pa^2_{ij}u)\cdot n_i,v_b\rangle_{\pa T}=0,
$$
since $v_{gi}$ and $v_b$ are single valued on each $e\in\E_h^0$, as well as $\pmb{v_g} =0$ and $v_b=0$ on $e\subset\pa\O.$

From \eqref{weak Hessian-1}-\eqref{weak Hessian-2}, $s=\min\{k,m,\ell,n\}$, the properties of the $L^2$ projection operators, Lemma \ref{pre-error-equation} with $\phi=u$ and $\varphi=\mathbb{Q}_s\pa_{ij,g}^2v\in P_s(T)$, there holds
\begin{equation}\label{error-equation-1-2}
\begin{split}
&\sum_{T\in{\cal T}_h}\sum_{i,j=1}^d(\mathbb{Q}_s\pa^2_{ij}u,\pa^2_{ij}v_0)_T\\
=&\sum_{T\in{\cal T}_h}\sum_{i,j=1}^d(\mathbb{Q}_s\pa^2_{ij}u,\pa^2_{ij,g}v)_T-(\mathbb{Q}_s\pa^2_{ij}u,\delta_gv)_T\\
=&\sum_{T\in{\cal T}_h}\sum_{i,j=1}^d(\mathbb{Q}_s\pa^2_{ij}u,\pa^2_{ij,g}v)_T
  -\langle(Q_bv_0-v_b)n_i,\pa_j(\mathbb{Q}_s\pa^2_{ij}u)\rangle_{\pa T}\\
 &+\langle Q_{gi}(\nabla v_0)-v_{gi},\mathbb{Q}_s\pa^2_{ij}u\cdot n_j\rangle_{\pa T}\\
=&\sum_{T\in{\cal T}_h}\sum_{i,j=1}^d(\pa^2_{ij}u,\mathbb{Q}_s\pa^2_{ij,g}v)_T
  -\langle(v_0-v_b)n_i,\pa_j(\mathbb{Q}_s\pa^2_{ij}u)\rangle_{\pa T}
 \\&  +\langle \pa_iv_0-v_{gi},\mathbb{Q}_s\pa^2_{ij}u\cdot n_j\rangle_{\pa T}\\
=&\sum_{T\in{\cal T}_h}\sum_{i,j=1}^d(\pa^2_{ij,g}Q_hu,\mathbb{Q}_s\pa^2_{ij,g}v)_T-(Q_0u-u,\pa^2_{ji}(\mathbb{Q}_s\pa_{ij,g}^2v))_T
 \\
 & -\langle(v_0-v_b)n_i,\pa_j(\mathbb{Q}_s\pa^2_{ij}u)\rangle_{\pa T} +\langle \pa_iv_0-v_{gi},\mathbb{Q}_s\pa^2_{ij}u\cdot n_j\rangle_{\pa T}.
\end{split}
\end{equation}
Substituting \eqref{error-equation-1-2} into \eqref{error-equation-0} gives
\begin{equation}\label{error-equation-2-1}
\begin{split}
(f,v_0)=&\sum_{T\in{\cal T}_h}\sum_{i,j=1}^d((I-\mathbb{Q}_s)\pa^2_{ij}u,\pa^2_{ij}v_0)_T+(\pa^2_{ij,g}Q_hu,(\mathbb{Q}_s-I)\pa^2_{ij,g}v)_T\\
&-(Q_0u-u,\pa^2_{ji}(\mathbb{Q}_s\pa_{ij,g}^2v))_T+\langle(v_0-v_b)n_i,\pa_j(\pa^2_{ij}u-\mathbb{Q}_s\pa^2_{ij}u)\rangle_{\pa T}\\
&+\langle \pa_iv_0-v_{gi},(\mathbb{Q}_s-I)\pa^2_{ij}u\cdot n_j\rangle_{\pa T}+(\pa_g^2Q_hu,\pa_g^2v)_{{\cal T}_h}.
\end{split}
\end{equation}
Subtracting \eqref{error-equation-2-1} by the {\rm gWG} scheme \eqref{WG-scheme}   gives rise to the error
equation \eqref{Error-equation}. This completes the proof.
\end{proof}

\section{Technical results}\label{technique-estimate}
  We shall provide some technical results in this section.

Let ${\cal T}_h$ be a finite element partition that satisfies the shape regular  assumption described as in \cite{ellip_MC2014}. The trace inequality holds true; i.e.,
\begin{equation}\label{trace inequality}
\|\phi\|_{\pa T}^2\lesssim h_T^{-1}\|\phi\|_T^2+h_T\|\nabla\phi\|_T^2,\qquad\forall\phi\in H^1(T).
\end{equation}
Moreover, for any  polynomial $\phi$, using the inverse inequality, the trace inequality \eqref{trace inequality} can be written as follows
\begin{equation}\label{inverse inequality}
\|\phi\|_{\pa T}^2\lesssim h_T^{-1}\|\phi\|_T^2.
\end{equation}

\begin{lemma}\cite{ellip_MC2014}\label{error projection}
Let ${\cal T}_h$ be a finite element partition satisfying the shape regular  assumption specified in \cite{ellip_MC2014}. Let $\alpha\in[0,k]$ and $\nu\in[0,s]$. For $0\leq t\leq2$, the following estimates hold true; i.e.,
\begin{eqnarray*}
&&\sum_{T\in{\cal T}_h}h_T^{2t}\|\phi-Q_0\phi\|_{t,T}^2\lesssim h^{2(\alpha+1)}\|\phi\|_{\alpha+1}^2,\\
&&\sum_{T\in{\cal T}_h}\sum_{i,j=1}^dh_T^{2t}\|\pa^2_{ij}\phi-\mathbb{Q}_s(\pa^2_{ij}\phi)\|_{t,T}^2\lesssim h^{2(\nu+1)}\|\phi\|_{\nu+3}^2.
\end{eqnarray*}
\end{lemma}

\begin{lemma}\label{strong-Hessian-estimate}
For any $v\in V_h$, there holds
$$
\Big(\sum_{T\in{\cal T}_h}\sum_{i,j=1}^d\|\pa^2_{ij}v_0\|_T^2\Big)^{\frac{1}{2}}\lesssim(1+h^{\frac{-3-\gamma_1}{2}}+h^{\frac{-1-\gamma_2}{2}})
 \3barv\3bar.
$$
\end{lemma}
\begin{proof}
From \eqref{weak Hessian-1} and the triangle inequality, there holds
\begin{equation}\label{strong-Hessian-estimate-1}
\begin{split}
\Big(\sum_{T\in{\cal T}_h}\sum_{i,j=1}^d\|\pa^2_{ij}v_0\|_T^2\Big)^{\frac{1}{2}}
=&\Big(\sum_{T\in{\cal T}_h}\sum_{i,j=1}^d\|\pa^2_{ij,g}v-\delta_gv\|_T^2\Big)^{\frac{1}{2}}\\
\lesssim&\3barv\3bar+\Big(\sum_{T\in{\cal T}_h}\|\delta_gv\|_T^2\Big)^{\frac{1}{2}}.
\end{split}
\end{equation}
We use \eqref{weak Hessian-2}, the Cauchy-Schwarz inequality, the trace inequality \eqref{inverse inequality} and the inverse inequality to obtain
\begin{equation*}
\begin{split}
\|\delta_gv\|_T
=&\sup_{\forall\varphi\in P_n(T)}\frac{|(\delta_gv,\varphi)_T|}{\|\varphi\|_T}\\
=&\sup_{\forall\varphi\in P_n(T)}\frac{|\langle(Q_bv_0-v_b)n_i,\pa_j\varphi\rangle_{\pa T}
  -\langle Q_{gi}(\nabla v_0)-v_{gi},\varphi n_j\rangle_{\pa T}|}{\|\varphi\|_T}\\
\lesssim&\sup_{\forall\varphi\in P_n(T)}\frac{\|Q_bv_0-v_b\|_{\pa T}\|\pa_j\varphi\|_{\pa T}
  +\|Q_{gi}(\nabla v_0)-v_{gi}\|_{\pa T}\|\varphi\|_{\pa T}}{\|\varphi\|_T}\\
  \lesssim& h_T^{\frac{-3}{2}}\|Q_bv_0-v_b\|_{\pa T}
  +h_T^{\frac{-1}{2}}\|Q_{gi}(\nabla v_0)-v_{gi}\|_{\pa T}.
\end{split}
\end{equation*}
This gives
\begin{equation*}
\begin{split}
&\Big(\sum_{T\in{\cal T}_h}\|\delta_gv\|_T^2\Big)^{\frac{1}{2}}\\
\lesssim&h_T^{\frac{-3-\gamma_1}{2}}\Big(\rho_1h_T^{\gamma_1}\|Q_bv_0-v_b\|^2_{\pa T}\Big)^\frac{1}{2}+h_T^{\frac{-1-\gamma_2}{2}}\Big(\rho_2h_T^{\gamma_2}\|Q_{gi}(\nabla v_0)-v_{gi}\|^2_{\pa T}\Big)^\frac{1}{2}\\
\lesssim&(h^{\frac{-3-\gamma_1}{2}}+h^{\frac{-1-\gamma_2}{2}})\3barv\3bar,
\end{split}
\end{equation*}
which, together with \eqref{strong-Hessian-estimate-1}, leads to Lemma \ref{strong-Hessian-estimate}. This completes the proof.
\end{proof}

\begin{lemma}\label{strong-gradient-estimate}
For any $v\in V_h^0$, there holds
$$
\Big(\sum_{T\in{\cal T}_h}\|\nabla v_0\|^2_T\Big)^{\frac{1}{2}}\lesssim(1+h^{\frac{-3-\gamma_1}{2}}+h^{\frac{-1-\gamma_2}{2}})\3barv\3bar.
$$
\end{lemma}
\begin{proof}
By using the Poincar\'{e} inequality (see Lemma A.4 in \cite{WW_bihar-2014}), the trace inequality \eqref{inverse inequality}, the property of $\pmb{Q_g}$ and Lemma \ref{strong-Hessian-estimate}, there yields
\begin{equation*}
\begin{split}
\sum_{T\in{\cal T}_h}\|\nabla v_0\|_T^2
\lesssim&\sum_{T\in{\cal T}_h}\Big(\|\nabla(\nabla v_0)\|_{T}^2+h_T^{-1}\|\nabla v_0-\pmb{v_g}\|_{\pa T}^2\Big)\\
\lesssim&\sum_{T\in{\cal T}_h}\Big(\sum_{i,j=1}^d\|\pa_{ij}^2v_0\|_T^2
          +h_T^{-1}\|\nabla v_0-\pmb{Q_g}(\nabla v_0)\|_{\pa T}^2
        +h_{T}^{-1}\|\pmb{Q_g}(\nabla v_0)-\pmb{v_g}\|_{\pa T}^2\Big)\\
\lesssim&\sum_{T\in{\cal T}_h}\sum_{i,j=1}^d\|\pa_{ij}^2v_0\|_T^2+\sum_{T\in{\cal T}_h}h_T^{-1}h_T^{-1}\|\nabla v_0-\pmb{Q_g}(\nabla v_0)\|_{T}^2\\
&+h^{-1-\gamma_2}\Big(\sum_{T\in{\cal T}_h}\rho_2h_T^{\gamma_2}\|\pmb{Q_g}(\nabla v_0)-\pmb{v_g}\|_{\pa T}^2\Big)\\
\lesssim&\sum_{T\in{\cal T}_h}\sum_{i,j=1}^d\|\pa_{ij}^2v_0\|_T^2+\sum_{T\in{\cal T}_h}\sum_{i,j=1}^dh_T^{-2}h_T^{2}\|\pa_{ij}^2v_0\|_T^2
+h^{-1-\gamma_2}\3barv\3bar^2\\
\lesssim&(1+h^{\frac{-3-\gamma_1}{2}}+h^{\frac{-1-\gamma_2}{2}})^2\3barv\3bar^2+h^{-1-\gamma_2}\3barv\3bar^2\\
\lesssim&(1+h^{\frac{-3-\gamma_1}{2}}+h^{\frac{-1-\gamma_2}{2}})^2\3barv\3bar^2.
\end{split}
\end{equation*}
 This completes the proof of the lemma.
\end{proof}

\begin{lemma}\label{gradient property}
For $\theta\in[2,k]$, there holds
$$
\|\delta_gQ_h\phi\|_T\lesssim h_T^{\theta-1}\|\phi\|_{\theta+1,T}.
$$
\end{lemma}
\begin{proof}
From \eqref{weak Hessian-2}, Cauchy-Schwarz inequality, the trace  inequalities \eqref{trace inequality} and \eqref{inverse inequality}, the inverse inequality, and Lemma \ref{error projection}, one obtains
\begin{equation*}
\begin{split}
 &\|\delta_gQ_h\phi\|_T\\
=&\sup_{\forall\varphi\in P_n(T)}\frac{|(\delta_gQ_h\phi,\varphi)_T|}{\|\varphi\|_T}\\
=&\sup_{\forall\varphi\in P_n(T)}\frac{|\langle(Q_b(Q_0\phi)-Q_b\phi)n_i,\pa_j\varphi\rangle_{\pa T}
    -\langle Q_{gi}(\nabla Q_0\phi)-Q_{gi}(\nabla\phi),\varphi n_j\rangle_{\pa T}|}{\|\varphi\|_T}\\
\lesssim&\sup_{\forall\varphi\in P_n(T)}\frac{\|Q_0\phi-\phi\|_{\pa T}\|\pa_j\varphi\|_{\pa T}
    +\|\nabla Q_0\phi-\nabla\phi\|_{\pa T}\|\varphi\|_{\pa T}}{\|\varphi\|_T}\\
\lesssim&\sup_{\forall\varphi\in P_n(T)}\frac{h_T^{\frac{-3}{2}}\|Q_0\phi-\phi\|_{\pa T}\|\varphi\|_{T}
    +h_T^{\frac{-1}{2}}\|\nabla Q_0\phi-\nabla\phi\|_{\pa T}\|\varphi\|_{T}}{\|\varphi\|_T}\\
\lesssim&h_T^{\frac{-3}{2}}\Big(h_T^{-1}\|Q_0\phi-\phi\|_{T}^2+h_T|Q_0\phi-\phi|_{1,T}^2\Big)^{\frac{1}{2}}+
    h_T^{\frac{-1}{2}}\Big(h_T^{-1}|Q_0\phi-\phi|_{1,T}^2+h_T|Q_0\phi-\phi|_{2,T}^2\Big)^{\frac{1}{2}}\\
\lesssim&h_T^{\frac{-3}{2}}h_T^{\frac{2 \theta+1}{2}}\|\phi\|_{\theta+1,T}+h_T^{\frac{-1}{2}}h_T^{\frac{2\theta-1}{2}}\|\phi\|_{\theta+1,T}\\
\lesssim&h_T^{\theta-1}\|\phi\|_{\theta+1,T}.
\end{split}
\end{equation*}
This  completes the proof of the lemma.
\end{proof}

\begin{lemma}\label{energy-estimate-pre} Let $\theta\in[2,k]$ and $\nu\in[0,s]$.  For any $v\in V_h$, there holds  
\begin{equation}\label{eq1}
\begin{split}
 |s(Q_h\phi,v)|\lesssim (h^{\frac{2\theta+\gamma_1+1}{2}}+h^{\frac{2\theta+\gamma_2-1}{2}})\|\phi\|_{\theta+1}\3barv\3bar,
\end{split}
\end{equation}
\begin{equation}\label{eq2}
\begin{split}
|\sum_{T\in{\cal T}_h}\sum_{i,j=1}^d(Q_0\phi-\phi,\pa^2_{ji}(\mathbb{Q}_s\pa^2_{ij,g}v))_T|\lesssim h^{\theta-1}\|\phi\|_{\theta+1}\3barv\3bar,\end{split}
\end{equation}
\begin{equation}\label{eq3}
\begin{split}
|\sum_{T\in{\cal T}_h}\sum_{i,j=1}^d(\pa^2_{ij}v_0,(\mathbb{Q}_s-I)\pa^2_{ij}\phi)_T|
  \lesssim (h^{\nu+1}+h^{\frac{2\nu-\gamma_1-1}{2}}+h^{\frac{2\nu-\gamma_2+1}{2}})\|\phi\|_{\nu+3}\3barv\3bar,\end{split}
\end{equation}
\begin{equation}\label{eq4}
\begin{split}
 &|\sum_{T\in{\cal T}_h}\sum_{i,j=1}^d\langle(v_0-v_b)n_i,\pa_j(\mathbb{Q}_s\pa^2_{ij}\phi-\pa^2_{ij}\phi)\rangle_{\pa T}|\\&=\begin{cases} (h+h^{\frac{-1-\gamma_1}{2}}+h^{\frac{1-\gamma_2}{2}})\|\phi\|_4\3barv\3bar,  \mbox{if}~ s=0,\\
(h^{\frac{2\nu-\gamma_1-1}{2}}+h^{\nu+1}+h^{\frac{2\nu-\gamma_2+1}{2}})\|\phi\|_{\nu+3}\3barv\3bar, \mbox{if}~ s>0,
         \end{cases}\end{split}
\end{equation}
\begin{equation}\label{eq5}
\begin{split}
&|\sum_{T\in{\cal T}_h}\sum_{i,j=1}^d\langle \pa_iv_0-v_{gi},(I-\mathbb{Q}_s)\pa^2_{ij}\phi\cdot n_j\rangle_{\pa T}|\\
  &\lesssim (h^{\frac{2\nu+1-\gamma_2}{2}}+h^{\frac{2\nu-\gamma_1-1}{2}}+h^{\nu+1})\|\phi\|_{\nu+3}\3barv\3bar,
\end{split}
\end{equation}
\begin{equation}\label{eq6}
\begin{split}|\sum_{T\in{\cal T}_h}(\pa_{ij,g}^2Q_h\phi,(I-\mathbb{Q}_s)\pa_{ij,g}^2v)_T|\lesssim(h^{\theta-1}\|\phi\|_{\theta+1}+h^{\nu+1}\|\phi\|_{\nu+3})\3barv\3bar.
\end{split}
\end{equation}
\end{lemma}
\begin{proof}
As to \eqref{eq1},  using the Cauchy-Schwarz inequality, the trace inequality \eqref{trace inequality} and Lemma \ref{error projection} gives
\begin{equation*}
\begin{split}
 |s(Q_h\phi,v)|
=&|\sum_{T\in{\cal T}_h}\rho_1h_T^{\gamma_1}\langle Q_b(Q_0\phi)-Q_b\phi,Q_bv_0-v_b\rangle_{\pa T}\\
 &+\sum_{T\in{\cal T}_h}\rho_2 h_T^{\gamma_2}\langle\pmb{Q_g}(\nabla Q_0\phi)-\pmb{Q_g}(\nabla\phi),\pmb{Q_g}(\nabla v_0)-\pmb{v_g}\rangle_{\pa T}|\\
\lesssim&\Big(\sum_{T\in{\cal T}_h}\rho_1h_T^{\gamma_1}\|Q_0\phi-\phi\|_{\pa T}^2\Big)^{\frac{1}{2}}
         \Big(\sum_{T\in {\cal T}_h}\rho_1h_T^{\gamma_1}\|Q_bv_0-v_b\|_{\pa T}^2\Big)^{\frac{1}{2}}\\
&+\Big(\sum_{T\in{\cal T}_h}\rho_2h_T^{\gamma_2}\|\nabla Q_0\phi-\nabla\phi\|_{\pa T}^2\Big)^{\frac{1}{2}}
         \Big(\sum_{T\in{\cal T}_h}\rho_2h_T^{\gamma_2}\|\pmb{Q_g}(\nabla v_0)-\pmb{v_g}\|_{\pa T}^2\Big)^{\frac{1}{2}}\\
\lesssim&\Big(\sum_{T\in{\cal T}_h}h_T^{\gamma_1-1}\|Q_0\phi-\phi\|_{T}^2+h_T^{\gamma_1+1}|Q_0\phi-\phi|_{1,T}^2\Big)^{\frac{1}{2}}\3barv\3bar\\
&+\Big(\sum_{T\in{\cal T}_h}h_T^{\gamma_2-1}|Q_0\phi-\phi|_{1,T}^2+h_T^{\gamma_2+1}|Q_0\phi-\phi|_{2,T}^2\Big)^{\frac{1}{2}}\3barv\3bar\\
\lesssim&(h^{\frac{\gamma_1-1}{2}}h^{\theta+1}\|\phi\|_{\theta+1}+h^{\frac{\gamma_2-1}{2}}h^{\theta}\|\phi\|_{\theta+1})\3barv\3bar\\
\lesssim&(h^{\frac{2\theta+\gamma_1+1}{2}}+h^{\frac{2\theta+\gamma_2-1}{2}})\|\phi\|_{\theta+1}\3barv\3bar.
\end{split}
\end{equation*}
As to \eqref{eq2},  from the Cauchy-Schwarz inequality, Lemma \ref{error projection} and the inverse inequality, one has
\begin{equation*}
\begin{split}
&|\sum_{T\in{\cal T}_h}\sum_{i,j=1}^d(Q_0\phi-\phi,\pa^2_{ji}(\mathbb{Q}_s\pa^2_{ij,g}v))_T|\\
\lesssim&\Big(\sum_{T\in{\cal T}_h}\|Q_0\phi-\phi\|_T^2\Big)^{\frac{1}{2}}
         \Big(\sum_{T\in{\cal T}_h}\sum_{i,j=1}^d\|\pa^2_{ji}(\mathbb{Q}_s\pa^2_{ij,g}v)\|_T^2\Big)^{\frac{1}{2}}\\
\lesssim&h^{\theta+1}\|\phi\|_{\theta+1}\cdot h^{-2}\Big(\sum_{T\in{\cal T}_h}\sum_{i,j=1}^d\|\pa^2_{ij,g}v\|_T^2\Big)^{\frac{1}{2}}\\
\lesssim&h^{\theta-1}\|\phi\|_{\theta+1}\3barv\3bar.
\end{split}
\end{equation*}
As to \eqref{eq3},   we have from the Cauchy-Schwarz inequality, Lemmas \ref{error projection}-\ref{strong-Hessian-estimate} that
\begin{equation*}
\begin{split}
&|\sum_{T\in{\cal T}_h}\sum_{i,j=1}^d(\pa^2_{ij}v_0,(\mathbb{Q}_s-I)\pa^2_{ij}\phi)_T|\\
\lesssim&\Big(\sum_{T\in{\cal T}_h}\sum_{i,j=1}^d\|\pa^2_{ij}v_0\|_T^2\Big)^{\frac{1}{2}}
          \Big(\sum_{T\in{\cal T}_h}\sum_{i,j=1}^d\|(\mathbb{Q}_s-I)\pa^2_{ij}\phi\|_T^2\Big)^{\frac{1}{2}}\\
\lesssim&h^{\nu+1}(1+h^{\frac{-3-\gamma_1}{2}}+h^{\frac{-1-\gamma_2}{2}})\|\phi\|_{\nu+3}\3barv\3bar.
\end{split}
\end{equation*}
As to \eqref{eq4} for the case of  $s>0$,  it follows from the Cauchy-Schwarz inequality, the trace inequalities \eqref{trace inequality}-\eqref{inverse inequality}, Lemmas  \ref{error projection}-\ref{strong-Hessian-estimate} that
\begin{equation*}
\begin{split}
&|\sum_{T\in{\cal T}_h}\sum_{i,j=1}^d\langle(v_0-v_b)n_i,\pa_j(\mathbb{Q}_s\pa^2_{ij}\phi-\pa^2_{ij}\phi)\rangle_{\pa T}|\\
\lesssim&\Big(\sum_{T\in{\cal T}_h}\rho_1h_T^{\gamma_1}\|v_0-v_b\|_{\pa T}^2\Big)^{\frac{1}{2}}
          \Big(\sum_{T\in{\cal T}_h}\sum_{i,j=1}^dh_T^{-\gamma_1}\|\pa_j(\mathbb{Q}_s\pa^2_{ij}\phi-\pa^2_{ij}\phi)\|_{\pa T}^2\Big)^{\frac{1}{2}}\\
\lesssim&\Big(\sum_{T\in{\cal T}_h}\rho_1h_T^{\gamma_1}\|Q_bv_0-v_b\|_{\pa T}^2+\rho_1 h_T^{\gamma_1}\|v_0-Q_bv_0\|_{\pa T}^2\Big)^{\frac{1}{2}}\\
&\cdot\Big(\sum_{T\in{\cal T}_h}\sum_{i,j=1}^dh_T^{-\gamma_1-1}|\mathbb{Q}_s\pa^2_{ij}\phi-\pa^2_{ij}\phi|_{1,T}^2
+h_T^{-\gamma_1+1}|\mathbb{Q}_s\pa^2_{ij}\phi-\pa^2_{ij}\phi|_{2,T}^2\Big)^{\frac{1}{2}}\\
\lesssim&\Big(\3barv\3bar^2+\sum_{T\in{\cal T}_h}h_T^{\gamma_1}h_T^{-1}\|v_0-Q_bv_0\|_{T}^2\Big)^{\frac{1}{2}}
\Big(\sum_{T\in{\cal T}_h}\sum_{i,j=1}^dh_T^{-\gamma_1-1}h_T^{2\nu}\|\phi\|_{\nu+3,T}^2\Big)^{\frac{1}{2}}\\
\lesssim&\Big(\3barv\3bar^2+\sum_{T\in{\cal T}_h}h_T^{\gamma_1-1}h_T^{4}|v_0|_{2,T}^2\Big)^{\frac{1}{2}}h^{\frac{2\nu-\gamma_1-1}{2}}\|\phi\|_{\nu+3}\\
\lesssim&\Big(\3barv\3bar+h^{\frac{\gamma_1+3}{2}}(1+h^{\frac{-3-\gamma_1}{2}}+h^{\frac{-1-\gamma_2}{2}})\3barv\3bar\Big)
h^{\frac{2\nu-\gamma_1-1}{2}}\|\phi\|_{\nu+3}\\
\lesssim&h^{\frac{2\nu-1-\gamma_1}{2}}(1+h^{\frac{3+\gamma_1}{2}}+h^{\frac{2+\gamma_1-\gamma_2}{2}})\|\phi\|_{\nu+3}\3barv\3bar.
\end{split}
\end{equation*}
As to \eqref{eq4}  for the case of $s=0$, we use the Cauchy-Schwarz inequality, the trace inequalities \eqref{trace inequality}-\eqref{inverse inequality} and Lemma
\ref{strong-gradient-estimate} to obtain
\begin{equation*}
\begin{split}
 &|\sum_{T\in{\cal T}_h}\sum_{i,j=1}^d\langle(v_0-v_b)n_i,\pa_j(\mathbb{Q}_s\pa^2_{ij}\phi-\pa^2_{ij}\phi)\rangle_{\pa T}|\\
=&|\sum_{T\in{\cal T}_h}\sum_{i,j=1}^d\langle(v_0-v_b)n_i,\pa_j(\pa^2_{ij}\phi)\rangle_{\pa T}|\\
=&|\sum_{T\in{\cal T}_h}\sum_{i,j=1}^d\langle(v_0-Q_bv_0)n_i,\pa_j(\pa^2_{ij}\phi)\rangle_{\pa T}
   +\langle(Q_bv_0-v_b)n_i,\pa_j(\pa^2_{ij}\phi)\rangle_{\pa T}|\\
=&|\sum_{T\in{\cal T}_h}\sum_{i,j=1}^d\langle(v_0-Q_bv_0)n_i,(I-Q_b)\pa_j(\pa^2_{ij}\phi)\rangle_{\pa T}
   +\langle(Q_bv_0-v_b)n_i,\pa_j(\pa^2_{ij}\phi)\rangle_{\pa T}|\\
\lesssim&\Big(\sum_{T\in{\cal T}_h}\|v_0-Q_bv_0\|_{\pa T}^2\Big)^{\frac{1}{2}}\Big(\sum_{T\in{\cal T}_h}\sum_{i,j=1}^d\|(I-Q_b)\pa_j(\pa^2_{ij}\phi)\|_{\pa T}^2\Big)^{\frac{1}{2}}\\
&+\Big(\sum_{T\in{\cal T}_h}\rho_1h_T^{\gamma_1}\|Q_bv_0-v_b\|_{\pa T}^2\Big)^{\frac{1}{2}}\Big(\sum_{T\in{\cal T}_h}\sum_{i,j=1}^dh_T^{-\gamma_1}\|\pa_j(\pa^2_{ij}\phi)\|_{\pa T}^2\Big)^{\frac{1}{2}}\\
\lesssim&\Big(\sum_{T\in{\cal T}_h}h_T^{-1}\|v_0-Q_bv_0\|_{T}^2\Big)^{\frac{1}{2}}\Big(\sum_{T\in{\cal T}_h}\sum_{i,j=1}^dh_T^{-1}\|(I-Q_b)\pa_j(\pa^2_{ij}\phi)\|_{T}^2+h_T|(I-Q_b)\pa_j(\pa^2_{ij}\phi)|_{1,T}^2\Big)^{\frac{1}{2}}\\
&+\3barv\3bar\Big(\sum_{T\in{\cal T}_h}h_T^{-\gamma_1}h_T^{-1}|\phi|_{3,T}^2+h_T^{-\gamma_1}h_T|\phi|_{4,T}^2\Big)^{\frac{1}{2}}\\
\lesssim&\Big(\sum_{T\in{\cal T}_h}h_T^{-1}h_T^2|v_0|_{1,T}^2\Big)^{\frac{1}{2}}\Big(\sum_{T\in{\cal T}_h}h_T^{-1}h_T^2|\phi|_{4,T}^2\Big)^{\frac{1}{2}}+h^{\frac{-1-\gamma_1}{2}}\|\phi\|_{4}\3barv\3bar\\
\lesssim&h(1+h^{\frac{-3-\gamma_1}{2}}+h^{\frac{-1-\gamma_2}{2}})\|\phi\|_4\3barv\3bar+h^{\frac{-1-\gamma_1}{2}}\|\phi\|_4\3barv\3bar\\
\lesssim&(h+h^{\frac{-1-\gamma_1}{2}}+h^{\frac{1-\gamma_2}{2}})\|\phi\|_4\3barv\3bar.
\end{split}
\end{equation*}
As to \eqref{eq5},  it follows from the Cauchy-Schwarz inequality, the trace inequalities  \eqref{trace inequality}-\eqref{inverse inequality}, Lemmas \ref{error projection}-\ref{strong-Hessian-estimate} that
\begin{equation*}
\begin{split}
&|\sum_{T\in{\cal T}_h}\sum_{i,j=1}^d\langle\pa_iv_0-v_{gi},(I-\mathbb{Q}_s)\pa^2_{ij}\phi\cdot n_j\rangle_{\pa T}|\\
\lesssim&\Big(\sum_{T\in{\cal T}_h}\|\nabla v_0-\pmb{Q_g}(\nabla v_0)\|_{\pa T}^2+\|\pmb{Q_g}(\nabla v_0)-\pmb{v_g}\|_{\pa T}^2\Big)^{\frac{1}{2}}
           \Big(\sum_{T\in{\cal T}_h}\sum_{i,j=1}^d\|(I-\mathbb{Q}_s)\pa^2_{ij}\phi\|_{\pa T}^2\Big)^{\frac{1}{2}}\\
\lesssim&\Big(\sum_{T\in{\cal T}_h}h_T^{-1}\|\nabla v_0-\pmb{Q_g}(\nabla v_0)\|_{T}^2
+h_T^{-\gamma_2}\rho_2 h_T^{\gamma_2}\|\pmb{Q_g}(\nabla v_0)-\pmb{v_g}\|_{\pa T}^2\Big)^{\frac{1}{2}}\\
&\cdot\Big(\sum_{T\in{\cal T}_h}\sum_{i,j=1}^dh_T^{-1}\|(I-\mathbb{Q}_s)\pa^2_{ij}\phi\|_{T}^2
+h_T|(I-\mathbb{Q}_s)\pa^2_{ij}\phi|_{1,T}^2\Big)^{\frac{1}{2}}\\
\lesssim&\Big(\sum_{T\in{\cal T}_h}h_T^{-1}h_T^2|v_0|_{2,T}^2+h^{-\gamma_2}\3barv\3bar^2\Big)^{\frac{1}{2}}
\Big(\sum_{T\in{\cal T}_h}h_T^{-1}h_T^{2(\nu+1)}\|\phi\|_{\nu+3}^2\Big)^{\frac{1}{2}}\\
\lesssim&\Big(h^{\frac{1}{2}}(1+h^{\frac{-3-\gamma_1}{2}}+h^{\frac{-1-\gamma_2}{2}})+h^{\frac{-\gamma_2}{2}}\Big) \3barv\3bar\cdot h^{\frac{2\nu+1}{2}}\|\phi\|_{\nu+3}\\
\lesssim&(h^{\frac{2\nu+1-\gamma_2}{2}}+h^{\frac{2\nu-\gamma_1-1}{2}}+h^{\nu+1})\|\phi\|_{\nu+3}\3barv\3bar.
\end{split}
\end{equation*}
As to \eqref{eq6},  using  \eqref{weak Hessian-1}, Lemma \ref{pre-error-equation} with $\varphi=\mathbb{Q}_s\pa_{ij,g}^2v\in P_s(T)$, the Cauchy-Schwarz inequality, Lemma \ref{error projection}, Lemma \ref{gradient property} and the inverse inequality, we have
\begin{equation*}
\begin{split}
 &|\sum_{T\in{\cal T}_h}(\pa_{ij,g}^2Q_h\phi,(I-\mathbb{Q}_s)\pa_{ij,g}^2v)_T|\\
=&|\sum_{T\in{\cal T}_h}\sum_{i,j=1}^d(\pa^2_{ij}Q_0\phi+\delta_gQ_h\phi,\pa_{ij,g}^2v)_T-(\pa_{ij,g}^2Q_h\phi,\mathbb{Q}_s\pa_{ij,g}^2v)_T|\\
=&|\sum_{T\in{\cal T}_h}\sum_{i,j=1}^d(\pa^2_{ij}Q_0\phi+\delta_gQ_h\phi,\pa_{ij,g}^2v)_T-(\pa^2_{ij}\phi,\mathbb{Q}_s\pa_{ij,g}^2v)_T
   -(Q_0\phi-\phi,\pa^2_{ji}(\mathbb{Q}_s\pa_{ij,g}^2v))_T|\\
=&|\sum_{T\in{\cal T}_h}\sum_{i,j=1}^d(\pa^2_{ij}Q_0\phi-\pa^2_{ij}\phi,\pa_{ij,g}^2v)_T+(\pa^2_{ij}\phi-\mathbb{Q}_s\pa^2_{ij}\phi,\pa_{ij,g}^2v)_T
+(\delta_gQ_h\phi,\pa_{ij,g}^2v)_T\\
 &-(Q_0\phi-\phi,\pa^2_{ji}(\mathbb{Q}_s\pa_{ij,g}^2v))_T|\\
\lesssim&\Big(\sum_{T\in{\cal T}_h}\sum_{i,j=1}^d\|\pa^2_{ij}Q_0\phi-\pa^2_{ij}\phi\|_T^2\Big)^{\frac{1}{2}}\3barv\3bar
+\Big(\sum_{T\in{\cal T}_h}\sum_{i,j=1}^d\|\pa^2_{ij}\phi-\mathbb{Q}_s\pa^2_{ij}\phi\|_T^2\Big)^{\frac{1}{2}}\3barv\3bar\\
&+\Big(\sum_{T\in{\cal T}_h}\|\delta_gQ_h\phi\|_T^2\Big)^{\frac{1}{2}}\3barv\3bar
+\Big(\sum_{T\in{\cal T}_h}\|Q_0\phi-\phi\|_T^2\Big)^{\frac{1}{2}}
    \Big(\sum_{T\in{\cal T}_h}\sum_{i,j=1}^d\|\pa^2_{ji}(\mathbb{Q}_s\pa_{ij,g}^2v)\|_T^2\Big)^{\frac{1}{2}}\\
\lesssim&(h^{\theta-1}\|\phi\|_{\theta+1}+h^{\nu+1}\|\phi\|_{\nu+3}+h^{\theta-1}\|\phi\|_{\theta+1})\3barv\3bar\\
&+h^{\theta+1}\|\phi\|_{\theta+1}h^{-2}\Big(\sum_{T\in{\cal T}_h}\sum_{i,j=1}^d\|\pa_{ij,g}^2v\|_T^2\Big)^{\frac{1}{2}}\\
\lesssim&(h^{\theta-1}\|\phi\|_{\theta+1}+h^{\nu+1}\|\phi\|_{\nu+3})\3barv\3bar.
\end{split}
\end{equation*}

This completes the proof of the lemma.
\end{proof}

\begin{lemma}\label{L2-estimate-pre}
 Let $t_0=\min\{k,3\}$. For any $\Psi\in H^4(\O)$, there holds 
\begin{equation}\label{q1}
\begin{split}
 |s(Q_hu,Q_h\Psi)|\lesssim (h^{k+t_0+\gamma_1+1}+h^{k+t_0+\gamma_2-1})\|u\|_{k+1}\|\Psi\|_{4},\\
\end{split}
\end{equation}
\begin{equation}\label{q2}
\begin{split} |\sum_{T\in{\cal T}_h}\sum_{i,j=1}^d(Q_0u-u,\pa^2_{ji}(\mathbb{Q}_s\pa_{ij,g}^2Q_h\Psi))_T|\lesssim (h^{k+1}+h^{k+t_0-2}) \|u\|_{k+1}\|\Psi\|_{4}, \\
\end{split}
\end{equation}
\begin{equation}\label{q3}
\begin{split} |\sum_{T\in{\cal T}_h}\sum_{i,j=1}^d(\pa^2_{ij}Q_0\Psi,(\mathbb{Q}_s-I)\pa^2_{ij}u)_T|\lesssim h^{s+t_0} \|u\|_{s+3}\|\Psi\|_{4},\\
\end{split}
\end{equation}
\begin{equation}\label{q4}
\begin{split} &|\sum_{T\in{\cal T}_h}\sum_{i,j=1}^d\langle(Q_0\Psi-Q_b\Psi)n_i,\pa_j(\mathbb{Q}_s-I)\pa^2_{ij}u\rangle_{\pa T}|\\
&\lesssim\begin{cases} h^{t_0}\|u\|_{4}\|\Psi\|_{4},& \mbox{if}~ s=0,\\
h^{s+t_0} \|u\|_{s+3}\|\Psi\|_{4},&  \mbox{if}~ s>0,
         \end{cases}\\
\end{split}
\end{equation}
\begin{equation}\label{q5}
\begin{split} |\sum_{T\in{\cal T}_h}\sum_{i,j=1}^d\langle\pa_iQ_0\Psi-Q_{gi}(\nabla\Psi),(I-\mathbb{Q}_s)\pa^2_{ij}u\cdot n_j\rangle_{\pa T}|\lesssim h^{s+t_0} \|u\|_{s+3}\|\Psi\|_{4}, 
\end{split}
\end{equation}
\begin{equation}\label{q6}
\begin{split} &|\sum_{T\in{\cal T}_h}\sum_{i,j=1}^d(\pa^2_{ij,g}Q_hu,(I-\mathbb{Q}_s)\pa_{ij,g}^{2}Q_h\Psi)_T|\\&\lesssim\begin{cases} (h^{k}\|u\|_{k+1}+h^{2}\|u\|_{3})(1+h^{t_0-2})\|\Psi\|_{4},& \mbox{if}~ s=0,\\ ((h^{k+1}+h^{k+t_0-2})\|u\|_{k+1}+(h^{s+3}+h^{s+t_0})\|u\|_{s+3})\|\Psi\|_{4},&  \mbox{if}~ s>0. 
         \end{cases}
\end{split}
\end{equation}
\end{lemma}

\begin{proof}
As to \eqref{q1}, using the Cauchy-Schwarz inequality, the trace inequality \eqref{trace inequality} and Lemma \ref{error projection}, there holds
\begin{equation*}\label{error-equation-3}
\begin{split}
 &|s(Q_hu,Q_h\Psi)|\\
=&|\sum_{T\in{\cal T}_h}\rho_1h_T^{\gamma_1}\langle Q_b(Q_0u)-Q_bu,Q_b(Q_0\Psi)-Q_b\Psi\rangle_{\pa T}\\
 &+\sum_{T\in{\cal T}_h}\rho_2h_T^{\gamma_2}\langle\pmb{Q_g}(\nabla Q_0u)-\pmb{Q_g}(\nabla u),\pmb{Q_g}(\nabla Q_0\Psi)-\pmb{Q_g}(\nabla\Psi)
  \rangle_{\pa T}|\\
\lesssim&\Big(\sum_{T\in{\cal T}_h}h_T^{\gamma_1}\|Q_0u-u\|_{\pa T}^2\Big)^{\frac{1}{2}}
     \Big(\sum_{T\in{\cal T}_h}h_T^{\gamma_1}\|Q_0\Psi-\Psi\|_{\pa T}^2\Big)^{\frac{1}{2}}\\
 &+\Big(\sum_{T\in{\cal T}_h}h_T^{\gamma_2}\|\nabla Q_0u-\nabla u\|_{\pa T}^2\Big)^{\frac{1}{2}}
    \Big(\sum_{T\in{\cal T}_h}h_T^{\gamma_2}\|\nabla Q_0\Psi-\nabla\Psi\|_{\pa T}^2\Big)^{\frac{1}{2}}\\
\lesssim&\Big(\sum_{T\in{\cal T}_h}h_T^{\gamma_1-1}h_T^{2(k+1)}\|u\|_{k+1,T}^2\Big)^{\frac{1}{2}}
     \Big(\sum_{T\in{\cal T}_h}h_T^{\gamma_1-1}h_T^{2(t_0+1)}\|\Psi\|_{t_0+1,T}^2\Big)^{\frac{1}{2}}\\
 &+\Big(\sum_{T\in{\cal T}_h}h_T^{\gamma_2-1}h_T^{-2}h_T^{2(k+1)}\|u\|_{k+1,T}^2\Big)^{\frac{1}{2}}
     \Big(\sum_{T\in{\cal T}_h}h_T^{\gamma_2-1}h_T^{-2}h_T^{2(t_0+1)}\|\Psi\|_{t_0+1,T}^2\Big)^{\frac{1}{2}}\\
\lesssim&h^{k+t_0+\gamma_1+1}\|u\|_{k+1}\|\Psi\|_{t_0+1}+h^{k+t_0+\gamma_2-1}\|u\|_{k+1}\|\Psi\|_{t_0+1}\\
\lesssim&(h^{k+t_0+\gamma_1+1}+h^{k+t_0+\gamma_2-1})\|u\|_{k+1}\|\Psi\|_{4}.
\end{split}
\end{equation*}
As to \eqref{q2}, we use the Cauchy-Schwarz inequality, Lemma \ref{error projection}, \eqref{weak Hessian-1}, the inverse inequality and Lemma \ref{gradient property} to obtain
\begin{equation*}\label{error-equation-4}
\begin{split}
&|\sum_{T\in{\cal T}_h}\sum_{i,j=1}^d(Q_0u-u,\pa^2_{ji}(\mathbb{Q}_s\pa_{ij,g}^2Q_h\Psi))_T|\\
\lesssim&\Big(\sum_{T\in{\cal T}_h}\|Q_0u-u\|_T^2\Big)^{\frac{1}{2}}
           \Big(\sum_{T\in{\cal T}_h}\sum_{i,j=1}^{d}\|\pa^2_{ji}(\mathbb{Q}_s\pa_{ij,g}^2Q_h\Psi)\|_T^2\Big)^{\frac{1}{2}}\\
\lesssim&h^{k+1}\|u\|_{k+1}\Big(\sum_{T\in{\cal T}_h}\sum_{i,j=1}^d|\pa^2_{ij,g}Q_h\Psi|_{2,T}^2\Big)^{\frac{1}{2}}\\
\lesssim&h^{k+1}\|u\|_{k+1}\Big(\sum_{T\in{\cal T}_h}\sum_{i,j=1}^d|\pa^2_{ij}Q_0\Psi|_{2,T}^2+|\delta_gQ_h\Psi|_{2,T}^2\Big)^{\frac{1}{2}}\\
\lesssim&h^{k+1}\|u\|_{k+1}\Big(\sum_{T\in{\cal T}_h}\sum_{i,j=1}^d|\pa^2_{ij}Q_0\Psi|_{2,T}^2+h_T^{-4}\|\delta_gQ_h\Psi\|_{T}^2\Big)^{\frac{1}{2}}\\
\lesssim&h^{k+1}\|u\|_{k+1}\Big(\|\Psi\|_{4}^2+\sum_{T\in{\cal T}_h}h_T^{-4}h_T^{2(t_0-1)}\|\Psi\|_{t_0+1,T}^2\Big)^{\frac{1}{2}}\\
\lesssim&h^{k+1}(1+h^{t_0-3})\|u\|_{k+1}\|\Psi\|_{4}.
\end{split}
\end{equation*}
As to \eqref{q3}, from the Cauchy-Schwarz inequality and Lemma \ref{error projection}, we get
\begin{equation*}\label{error-equation-5}
\begin{split}
 &|\sum_{T\in{\cal T}_h}\sum_{i,j=1}^d(\pa^2_{ij}Q_0\Psi,(\mathbb{Q}_s-I)\pa^2_{ij}u)_T|\\
=&|\sum_{T\in{\cal T}_h}\sum_{i,j=1}^d((I-\mathbb{Q}_s)\pa^2_{ij}Q_0\Psi,(\mathbb{Q}_s-I)\pa^2_{ij}u)_T|\\
\lesssim&\Big(\sum_{T\in{\cal T}_h}\sum_{i,j=1}^d\|(I-\mathbb{Q}_s)\pa^2_{ij}Q_0\Psi\|_T^2\Big)^{\frac{1}{2}}
\Big(\sum_{T\in{\cal T}_h}\sum_{i,j=1}^d\|(\mathbb{Q}_s-I)\pa^2_{ij}u\|_T^2\Big)^{\frac{1}{2}}\\
\lesssim&h^{t_0-1}\|Q_0\Psi\|_{t_0+1}\cdot h^{s+1}\|u\|_{s+3}\\
\lesssim&h^{s+t_0}\|u\|_{s+3}\|\Psi\|_{4}.
\end{split}
\end{equation*}
As to \eqref{q4} for the case of  $s>0$,  from the Cauchy-Schwarz inequality, the trace inequality \eqref{trace inequality} and Lemma \ref{error projection}, we have
\begin{equation*}\label{error-equation-6-1}
\begin{split}
    &|\sum_{T\in{\cal T}_h}\sum_{i,j=1}^d\langle(Q_0\Psi-Q_b\Psi)n_i,\pa_j(\mathbb{Q}_s-I)\pa^2_{ij}u\rangle_{\pa T}|\\
\lesssim&\Big(\sum_{T\in{\cal T}_h}\|Q_0\Psi-\Psi\|_{\pa T}^2\Big)^{\frac{1}{2}}
        \Big(\sum_{T\in{\cal T}_h}\sum_{i,j=1}^d\|\pa_j(\mathbb{Q}_s-I)\pa^2_{ij}u\|_{\pa T}^2\Big)^{\frac{1}{2}}\\
\lesssim&\Big(\sum_{T\in{\cal T}_h}h_T^{-1}\|Q_0\Psi-\Psi\|_{T}^2+h_T\|\nabla(Q_0\Psi-\Psi)\|_{T}^2\Big)^{\frac{1}{2}}\\
&\cdot\Big(\sum_{T\in{\cal T}_h}\sum_{i,j=1}^dh_T^{-1}|(\mathbb{Q}_s-I)\pa^2_{ij}u|_{1,T}^2
+h_T|(\mathbb{Q}_s-I)\pa^2_{ij}u|_{2,T}^2\Big)^{\frac{1}{2}}\\
\lesssim&\Big(\sum_{T\in{\cal T}_h}h_T^{-1}h_T^{2(t_0+1)}\|\Psi\|_{t_0+1,T}^2\Big)^{\frac{1}{2}}
   \Big(\sum_{T\in{\cal T}_h}h_T^{-1}h_T^{-2}h_T^{2s+2}\|u\|_{s+3,T}^2\Big)^{\frac{1}{2}}\\
\lesssim&h^{t_0+s}\|u\|_{s+3}\|\Psi\|_{4}.
\end{split}
\end{equation*}
As to \eqref{q4}  for  the case of $s=0$, it follows from the Cauchy-Schwarz inequality, the trace inequality
\eqref{trace inequality} and Lemma \ref{error projection} that
\begin{equation*}\label{error-equation-6-2}
\begin{split}
    &|\sum_{T\in{\cal T}_h}\sum_{i,j=1}^d\langle(Q_0\Psi-Q_b\Psi)n_i,\pa_j(\mathbb{Q}_s-I)\pa^2_{ij}u\rangle_{\pa T}|\\
=&|\sum_{T\in{\cal T}_h}\sum_{i,j=1}^d\langle(Q_0\Psi-Q_b\Psi)n_i,\pa_j\pa^2_{ij}u\rangle_{\pa T}|\\
\lesssim&\Big(\sum_{T\in{\cal T}_h}\|Q_0\Psi-\Psi\|_{\pa T}^2\Big)^{\frac{1}{2}}
        \Big(\sum_{T\in{\cal T}_h}\sum_{i,j=1}^d\|\pa_j\pa^2_{ij}u\|_{\pa T}^2\Big)^{\frac{1}{2}}\\
\lesssim&\Big(\sum_{T\in{\cal T}_h}h_T^{-1}\|Q_0\Psi-\Psi\|_{T}^2+h_T|Q_0\Psi-\Psi|_{1,T}^2\Big)^{\frac{1}{2}}
\Big(\sum_{T\in{\cal T}_h}h_T^{-1}|u|_{3,T}^2+h_T|u|_{4,T}^2\Big)^{\frac{1}{2}}\\
\lesssim&\Big(\sum_{T\in{\cal T}_h}h_T^{-1}h_T^{2(t_0+1)}\|\Psi\|_{t_0+1,T}^2\Big)^{\frac{1}{2}}
   \Big(\sum_{T\in{\cal T}_h}h_T^{-1}\|u\|_{4,T}^2\Big)^{\frac{1}{2}}\\
\lesssim&h^{t_0}\|u\|_{4}\|\Psi\|_{4}.
\end{split}
\end{equation*}
As to \eqref{q5}, we have from the Cauchy-Schwarz inequality, the trace inequality  \eqref{trace inequality} and Lemma \ref{error projection} that
\begin{equation*}\label{error-equation-66}
\begin{split}
    &|\sum_{T\in{\cal T}_h}\sum_{i,j=1}^d\langle\pa_iQ_0\Psi-Q_{gi}(\nabla\Psi),(I-\mathbb{Q}_s)\pa^2_{ij}u\cdot n_j\rangle_{\pa T}|\\
\lesssim&\Big(\sum_{T\in{\cal T}_h}\|\nabla Q_0\Psi-\nabla\Psi\|_{\pa T}^2\Big)^{\frac{1}{2}}
        \Big(\sum_{T\in{\cal T}_h}\sum_{i,j=1}^d\|(I-\mathbb{Q}_s)\pa^2_{ij}u\|_{\pa T}^2\Big)^{\frac{1}{2}}\\
\lesssim&\Big(\sum_{T\in{\cal T}_h}h_T^{-1}|Q_0\Psi-\Psi|_{1,T}^2+h_T|Q_0\Psi-\Psi|_{2,T}^2\Big)^{\frac{1}{2}}\\
&\cdot\Big(\sum_{T\in{\cal T}_h}\sum_{i,j=1}^dh_T^{-1}\|(I-\mathbb{Q}_s)\pa^2_{ij}u\|_{T}^2+h_T|(I-\mathbb{Q}_s)\pa^2_{ij}u|_{1,T}^2\Big)^{\frac{1}{2}}\\
\lesssim&\Big(\sum_{T\in{\cal T}_h}h_T^{-1}h_T^{-2}h_T^{2(t_0+1)}\|\Psi\|_{t_0+1,T}^2\Big)^{\frac{1}{2}}
   \Big(\sum_{T\in{\cal T}_h}h_T^{-1}h_T^{2s+2}\|u\|_{s+3,T}^2\Big)^{\frac{1}{2}}\\
\lesssim&h^{t_0+s}\|u\|_{s+3}\|\Psi\|_{4}.
\end{split}
\end{equation*}
As to \eqref{q6}, let $t_1=\min\{s,1\}$. Using \eqref{weak Hessian-1}, the Cauchy-Schwarz inequality, Lemma  \ref{error projection},  Lemma \ref{gradient property}, and the inverse inequality  yields
\begin{equation*}\label{error-equation-7}
\begin{split}
 &|\sum_{T\in{\cal T}_h}\sum_{i,j=1}^d(\pa^2_{ij,g}Q_hu,(I-\mathbb{Q}_s)\pa_{ij,g}^{2}Q_h\Psi)_T|\\
=&|\sum_{T\in{\cal T}_h}\sum_{i,j=1}^d(\pa^2_{ij}Q_0u+\delta_{g}Q_hu,(I-\mathbb{Q}_s)\pa_{ij,g}^{2}Q_h\Psi)_{T}|\\
=&|\sum_{T\in{\cal T}_h}\sum_{i,j=1}^d(\pa^2_{ij}(Q_0u-u)+(I-\mathbb{Q}_s)\pa_{ij}^2u+\delta_gQ_hu,(I-\mathbb{Q}_s)\pa_{ij,g}^2Q_h\Psi)_T|\\
\lesssim&\Big(\sum_{T\in{\cal T}_h}\sum_{i,j=1}^d\|\pa^2_{ij}(Q_0u-u)\|_T^2
          +\|(I-\mathbb{Q}_s)\pa_{ij}^2u\|_T^2+\|\delta_gQ_hu\|_T^2\Big)^{\frac{1}{2}}\\
&\cdot\Big(\sum_{T\in{\cal T}_h}\sum_{i,j=1}^d\|(I-\mathbb{Q}_s)\pa_{ij,g}^2Q_h\Psi\|_T^2\Big)^{\frac{1}{2}}\\
\lesssim&\Big(\sum_{T\in{\cal T}_h}h_T^{2(k-1)}\|u\|_{k+1,T}^2+h_T^{2(s+1)}\|u\|_{s+3,T}^2\Big)^{\frac{1}{2}}
          \Big(\sum_{T\in{\cal T}_h}\sum_{i,j=1}^dh_T^{2(t_1+1)}|\pa_{ij,g}^{2}Q_h\Psi|_{t_1+1,T}^2\Big)^{\frac{1}{2}}\\
\lesssim&(h^{k-1}\|u\|_{k+1}+h^{s+1}\|u\|_{s+3})h^{t_1+1}\Big(\sum_{T\in{\cal T}_h}\sum_{i,j=1}^d|\pa^2_{ij}Q_0\Psi
          +\delta_gQ_h\Psi|_{t_1+1,T}^2\Big)^{\frac{1}{2}}\\
\lesssim&(h^{k-1}\|u\|_{k+1}+h^{s+1}\|u\|_{s+3})h^{t_1+1}\Big(\sum_{T\in{\cal T}_h}\sum_{i,j=1}^d|\pa^2_{ij}Q_0\Psi|_{t_1+1,T}^2
          +|\delta_gQ_h\Psi|_{t_1+1,T}^2\Big)^{\frac{1}{2}}\\
\lesssim&(h^{k+t_1}\|u\|_{k+1}+h^{s+t_1+2}\|u\|_{s+3})\\
&\cdot\Big(\sum_{T\in{\cal T}_h}\|\Psi\|_{4,T}^2+h_T^{2t_0-4}\delta_{t_1,0}\|\Psi\|_{t_0+1,T}^2+h_T^{2t_0-6}\delta_{t_1,1}\|\Psi\|_{t_0+1,T}^2\Big)^{\frac{1}{2}}\\
\lesssim&(h^{k+t_1}\|u\|_{k+1}+h^{s+t_1+2}\|u\|_{s+3})(1+h^{t_0-2}\delta_{t_1,0}+h^{t_0-3}\delta_{t_1,1})\|\Psi\|_{4}\\\lesssim&\begin{cases} (h^{k}\|u\|_{k+1}+h^{2}\|u\|_{3})(1+h^{t_0-2})\|\Psi\|_{4},& \mbox{if}~ s=0,\\
 ((h^{k+1}+h^{k+t_0-2})\|u\|_{k+1}+(h^{s+3}+h^{s+t_0})\|u\|_{s+3})\|\Psi\|_{4},&  \mbox{if}~ s>0, 
         \end{cases}
\end{split}
\end{equation*}
 where $\delta_{t_1,i}$ for $i=0,1$ is the usual Kronecker's delta with value 1 when $t_1=i$ for $i=0,1$ and value 0 otherwise.  This completes the proof of the lemma.
\end{proof}

\section{Error estimates in $H^2$}\label{error estimate}

The goal of this section is to establish some error estimates for the numerical approximation arising from the {\rm g}WG scheme \eqref{WG-scheme}.

\begin{theorem}\label{THM:energy-estimate}
Let $k\geq2$, $s=\min\{k,m,\ell,n\}$ and $q=\max\{k+1,s+3\}$. Let $u_h\in V_h$ be the numerical approximation arising from the {\rm gWG} scheme \eqref{WG-scheme}. Assume that the exact solution of the model equation \eqref{model-problem} satisfies $u\in H^{k+1}\cap H^4(\O)$ for $s=0$ and $u\in H^q(\O)$ for $s>0$. Then, the following error estimate holds true:
\begin{equation*}
\begin{split}
\3bare_h\3bar\lesssim
\left\{
\begin{array}{lr}
C_1(h,k,\gamma_1,\gamma_2)(\|u\|_{k+1}+\delta_{k,2}\|u\|_4),\quad if~s=0,\\
C_2(h,k,\gamma_1,\gamma_2,s)\|u\|_q,\;~\qquad\qquad\qquad if~s>0,
\end{array}
\right.
\end{split}
\end{equation*}
where $C_1(h,k,\gamma_1,\gamma_2)=h^{k-1}(h^{\frac{\gamma_1+3}{2}}+h^{\frac{\gamma_2+1}{2}})+h(h^{\frac{-\gamma_1-3}{2}}+h^{\frac{-\gamma_2-1}{2}}+1)$ and $C_2(h,k,\gamma_1,\gamma_2,s)=h^{k-1}(h^{\frac{\gamma_1+3}{2}}+h^{\frac{\gamma_2+1}{2}}+1)+h^{s+1}(h^{\frac{-\gamma_1-3}{2}}+h^{\frac{-\gamma_2-1}{2}}+1)$. 
\end{theorem}

\begin{proof} 
By setting $v=e_h\in V_h^0$ in the error equation \eqref{Error-equation} and combining \eqref{triple-bar}, one arrives at

\begin{equation*}\label{energy-estimate-2}
\3bare_h\3bar^2=\zeta_u(e_h).
\end{equation*}
It follows from \eqref{error equation-remainder}, Lemma \ref{energy-estimate-pre} with $\phi=u$, $v=e_h$, $\theta=k$, $\nu=s$ and $k\geq2$ that
\begin{equation*}
\begin{split}
\3bare_h\3bar\lesssim&
\left\{
\begin{array}{lr}
(h^{\frac{2k+\gamma_1+1}{2}}+h^{\frac{2k+\gamma_2-1}{2}}+h^{k-1})\|u\|_{k+1}+(h+h^{\frac{-1-\gamma_1}{2}}+h^{\frac{1-\gamma_2}{2}})\|u\|_4,~~ if~s=0,\\
(h^{\frac{2k+\gamma_1+1}{2}}+h^{\frac{2k+\gamma_2-1}{2}}+h^{k-1}+h^{s+1}+h^{\frac{2s-\gamma_1-1}{2}}+h^{\frac{2s-\gamma_2+1}{2}})\|u\|_q,if~s>0,
\end{array}
\right.\\
\lesssim&
\left\{
\begin{array}{lr}
(h^{\frac{2k+\gamma_1+1}{2}}+h^{\frac{2k+\gamma_2-1}{2}}+h+h^{\frac{-1-\gamma_1}{2}}+h^{\frac{1-\gamma_2}{2}})\|u\|_4,
~if~s=0,~k=2,\\
(h^{\frac{2k+\gamma_1+1}{2}}+h^{\frac{2k+\gamma_2-1}{2}}+h+h^{\frac{-1-\gamma_1}{2}}+h^{\frac{1-\gamma_2}{2}})\|u\|_{k+1},
~if~s=0,~k>2,\\
(h^{\frac{2k+\gamma_1+1}{2}}+h^{\frac{2k+\gamma_2-1}{2}}+h^{k-1}+h^{s+1}+h^{\frac{2s-\gamma_1-1}{2}}+h^{\frac{2s-\gamma_2+1}{2}})\|u\|_q,if~s>0,
\end{array}
\right.\\
\lesssim&
\left\{
\begin{array}{lr}
C_1(h,k,\gamma_1,\gamma_2)(\|u\|_{k+1}+\delta_{k,2}\|u\|_4),\quad if~s=0,\\
C_2(h,k,\gamma_1,\gamma_2,s)\|u\|_q,\;~\qquad\qquad\qquad if~s>0,
\end{array}
\right.
\end{split}
\end{equation*}
which gives rise to the desired result.  This completes the proof of the theorem.
\end{proof}

  \begin{remark} For the {\rm gWG} element $V_{k,m,\ell}(T)$   with $m=k-2$, $\ell=k-2$ and $n\leq k-2$,  $\3bare_h\3bar$ converges at an optimal rate of ${\cal O}(h^{k-1})$ when $\gamma_1=-3$ and $\gamma_2=-1$. 
 \end{remark}

\section{Error estimates in $L^2$}\label{errorl2}

This section shall establish an error estimate for the numerical approximation in the usual $L^2$ norm by using the standard duality argument technique. To this end, we consider the following dual problem
\begin{equation}\label{dual-equation}
\begin{split}
\Delta^2\Psi&=e_0,\quad\mbox{in}~~\O,\\
        \Psi&=0,~\quad\mbox{on}~~\pa\O,\\
\frac{\pa\Psi}{\pa\textbf{n}}&=0,~\quad\mbox{on}~~\pa\O.
\end{split}
\end{equation}
Assume that the dual problem \eqref{dual-equation} satisfies the $H^4$ regularity property in the sense that there exists a generic constant $C$ such that
\begin{equation}\label{dual-regular}
\|\Psi\|_4\leq C\|e_0\|.
\end{equation}

\begin{theorem}\label{THM:L2-estimate}
 Let $k\geq2$, $s=\min\{k,m,\ell,n\}$, $q=\max\{k+1,s+3\}$ and $t_0=\min\{k,3\}$. Let $u_h\in V_h$ be the numerical approximation of the {\rm gWG} scheme \eqref{WG-scheme}. Assume that the exact solution of the biharmonic equation \eqref{model-problem} is sufficiently regular so that $u\in H^{k+1}(\O)\cap H^4(\O)$ for $s=0$ and $u\in H^q(\O)$ for $s>0$.  In addition, assume that the dual problem \eqref{dual-equation} has the $H^4$ regularity property \eqref{dual-regular}. Then, the following error estimate holds true; i.e.,
 
\begin{equation*}
\begin{split}
\|e_0\|\lesssim
\left\{
\begin{array}{lr}
C_3(h,k,\gamma_1,\gamma_2,t_0)(\|u\|_{k+1}+\delta_{k,2}\|u\|_4),\quad if~s=0,\\
C_4(h,k,\gamma_1,\gamma_2,s,t_0)\|u\|_q,\;~\qquad\qquad\qquad if~s>0,
\end{array}
\right.
\end{split}
\end{equation*}
where $C_3(h,k,\gamma_1,\gamma_2,t_0)=(h^{t_0-1}(h^{\frac{\gamma_1+3}{2}}+h^{\frac{\gamma_2+1}{2}})
+h(1+h^{\frac{-\gamma_1-3}{2}}+h^{\frac{-\gamma_2-1}{2}}))C_1(h,k,\gamma_1,\gamma_2)$ and $C_4(h,k,\gamma_1,\gamma_2,s,t_0)=h^{t_0-1}(h^{\frac{\gamma_1+3}{2}}+h^{\frac{\gamma_2+1}{2}}+1+h^{\frac{-\gamma_1-3}{2}}+h^{\frac{-\gamma_2-1}{2}})
C_2(h,k,\gamma_1,\gamma_2,s)$. 
\end{theorem}

\begin{proof}
Testing the dual problem \eqref{dual-equation} against $e_0$ and using the usual integration by parts, we have
\begin{equation}\label{THM:L2-estimate-1}
\begin{split}
\|e_0\|^2
=&(\Delta^2\Psi,e_0)\\
=&\sum_{T\in{\cal T}_h}\sum_{i,j=1}^d(\pa^2_{ij}\Psi,\pa^2_{ij}e_0)_T+\langle\pa_j(\pa^2_{ij}\Psi)\cdot n_i,e_0\rangle_{\pa T}
   -\langle\pa^2_{ij}\Psi,\pa_ie_0\cdot n_j\rangle_{\pa T}\\
=&\sum_{T\in{\cal T}_h}\sum_{i,j=1}^d((I-\mathbb{Q}_s)\pa^2_{ij}\Psi,\pa^2_{ij}e_0)_T+(\mathbb{Q}_s\pa^2_{ij}\Psi,\pa^2_{ij}e_0)_T
   \\
 &+\langle\pa_j(\pa^2_{ij}\Psi)\cdot n_i,e_0-e_b\rangle_{\pa T} -\langle\pa^2_{ij}\Psi,(\pa_ie_0-e_{gi})\cdot n_j\rangle_{\pa T},
\end{split}
\end{equation}
where we used  $\sum_{T\in{\cal T}_h}\sum_{i,j=1}^d\langle\pa_j(\pa^2_{ij}\Psi)\cdot n_i,e_b\rangle_{\pa T}=0$ and $\sum_{T\in{\cal T}_h}\sum_{i,j=1}^d\langle\pa^2_{ij}\Psi,e_{gi}\cdot n_j\rangle_{\pa T}=0$, since $e_b=0$ and $\pmb{e_{g}} =0$ on $\partial\Omega$.

Letting $u=\Psi$ and $v=e_h$ in \eqref{error-equation-1-2} and letting $v=Q_h\Psi$ in the error equation \eqref{Error-equation}  to obtain
\begin{equation*}
\begin{split}
 &\sum_{T\in{\cal T}_h}\sum_{i,j=1}^d(\mathbb{Q}_s\pa^2_{ij}\Psi,\pa^2_{ij}e_0)_T\\
=&\sum_{T\in{\cal T}_h}\sum_{i,j=1}^d(\pa^2_{ij,g}Q_h\Psi,\mathbb{Q}_s\pa^2_{ij,g}e_h)_T-(Q_0\Psi-\Psi,\pa^2_{ji}(\mathbb{Q}_s\pa^2_{ij,g}e_h))_T\\
&-\langle(e_0-e_b)n_i,\pa_j(\mathbb{Q}_s\pa^2_{ij}\Psi)\rangle_{\pa T}
+\langle \pa_ie_0-e_{gi},\mathbb{Q}_s\pa^2_{ij}\Psi\cdot n_j\rangle_{\pa T}\\
=&\sum_{T\in{\cal T}_h}\sum_{i,j=1}^d(\pa^2_{ij,g}Q_h\Psi,(\mathbb{Q}_s-I)\pa^2_{ij,g}e_h)_T+\zeta_u(Q_h\Psi)-s(e_h,Q_h\Psi)
-(Q_0\Psi-\Psi,\pa^2_{ji}(\mathbb{Q}_s\pa^2_{ij,g}e_h))_T\\
&-\langle(e_0-e_b)n_i,\pa_j(\mathbb{Q}_s\pa^2_{ij}\Psi)\rangle_{\pa T}
+\langle \pa_ie_0-e_{gi},\mathbb{Q}_s\pa^2_{ij}\Psi\cdot n_j\rangle_{\pa T}.
\end{split}
\end{equation*}
Substituting the above equation into \eqref{THM:L2-estimate-1} gives rise to
\begin{equation}\label{error-equation-2}
\|e_0\|^2=\zeta_u(Q_h\Psi)-\zeta_\Psi(e_h).
\end{equation}
 
Next, it suffices to estimate the two terms on the right hand of \eqref{error-equation-2}. Let us first estimate the first term $\zeta_u(Q_h\Psi)$. For the case of $s=0$, it follows from \eqref{error equation-remainder}, Lemma \ref{L2-estimate-pre},  and \eqref{dual-regular} that
\begin{equation}\label{error-equation-3310}
\begin{split}
&|\zeta_u(Q_h\Psi)|\\ \lesssim&
((h^{k+t_0+\gamma_1+1}+h^{k+t_0+\gamma_2-1}+h^{k+1}+h^{k+t_0-2}+h^k)\|u\|_{k+1}\\&+h^{t_0}\|u\|_4+h^2\|u\|_3)\|\Psi\|_4,\\
\lesssim&((h^{k+t_0+\gamma_1+1}+h^{k+t_0+\gamma_2-1}+h^k)\|u\|_{k+1}+h^{t_0}\|u\|_4+h^2\|u\|_4)\|\Psi\|_4,\\
\lesssim&((h^{k+t_0+\gamma_1+1}+h^{k+t_0+\gamma_2-1}+h^k)\|u\|_{k+1}+h^2\|u\|_4)\|\Psi\|_4,\\
\lesssim&(h^{k+t_0+\gamma_1+1}+h^{k+t_0+\gamma_2-1}+h^2)(\|u\|_{k+1}+\delta_{k,2}\|u\|_4)\|e_0\|.
\end{split}
\end{equation}
For the case of $s>0$, using \eqref{error equation-remainder}, Lemma \ref{L2-estimate-pre},  and \eqref{dual-regular}  gives
\begin{equation}\label{error-equation-3311}
\begin{split}
&|\zeta_u(Q_h\Psi)|\\ \lesssim&
((h^{k+t_0+\gamma_1+1}+h^{k+t_0+\gamma_2-1}+h^{k+1}+h^{k+t_0-2})\|u\|_{k+1}\\
&+(h^{s+t_0}+h^{s+3})\|u\|_{s+3})\|\Psi\|_4,\\
\lesssim&
((h^{k+t_0+\gamma_1+1}+h^{k+t_0+\gamma_2-1}+h^{k+t_0-2})\|u\|_{k+1}+h^{s+t_0}\|u\|_{s+3})\|\Psi\|_4,\\
\lesssim&
(h^{k+t_0+\gamma_1+1}+h^{k+t_0+\gamma_2-1}+h^{k+t_0-2}+h^{s+t_0})\|u\|_q\|\Psi\|_4,\\
\lesssim&
(h^{k+t_0+\gamma_1+1}+h^{k+t_0+\gamma_2-1}+h^{k+t_0-2}+h^{s+t_0})\|u\|_q\|e_0\|.
\end{split}
\end{equation}

As to the second term $\zeta_\Psi(e_h)$, for the case of $s=0$, using \eqref{error equation-remainder}, Lemma \ref{energy-estimate-pre} with $\phi=\Psi$, $v=e_h$, $\theta=t_0$, $\nu=0$, \eqref{dual-regular} and Theorem \ref{THM:energy-estimate} gives
\begin{equation}\label{error-equation-3312}
\begin{split}
|\zeta_\Psi(e_h)|
\lesssim&((h^{\frac{2t_0+\gamma_1+1}{2}}+h^{\frac{2t_0+\gamma_2-1}{2}}+h^{t_0-1})
  \|\Psi\|_{t_0+1}\\
&+(h+h^{\frac{-1-\gamma_1}{2}}+h^{\frac{1-\gamma_2}{2}})\|\Psi\|_{4})\3bare_h\3bar\\
\lesssim&(h^{\frac{2t_0+\gamma_1+1}{2}}+h^{\frac{2t_0+\gamma_2-1}{2}}+h^{t_0-1}+h
+h^{\frac{-1-\gamma_1}{2}}+h^{\frac{1-\gamma_2}{2}})\|\Psi\|_{4}\3bare_h\3bar\\
\lesssim&(h^{\frac{2t_0+\gamma_1+1}{2}}+h^{\frac{2t_0+\gamma_2-1}{2}}
+h+h^{\frac{-1-\gamma_1}{2}}+h^{\frac{1-\gamma_2}{2}})\|\Psi\|_{4}\3bare_h\3bar\\
\lesssim&(h^{t_0-1}(h^{\frac{\gamma_1+3}{2}}+h^{\frac{\gamma_2+1}{2}})
+h(1+h^{\frac{-\gamma_1-3}{2}}+h^{\frac{-\gamma_2-1}{2}}))\|e_0\|\3bare_h\3bar,\\
\lesssim&(h^{t_0-1}(h^{\frac{\gamma_1+3}{2}}+h^{\frac{\gamma_2+1}{2}})
+h(1+h^{\frac{-\gamma_1-3}{2}}+h^{\frac{-\gamma_2-1}{2}}))C_1(h,k,\gamma_1,\gamma_2)\\
&\cdot(\|u\|_{k+1}+\delta_{k,2}\|u\|_4)\|e_0\|\\
\lesssim&C_3(h,k,\gamma_1,\gamma_2,t_0)(\|u\|_{k+1}+\delta_{k,2}\|u\|_4)\|e_0\|.
\end{split}
\end{equation}
For the case of $s>0$, we apply \eqref{error equation-remainder}, Lemma \ref{energy-estimate-pre} with $\phi=\Psi$, $v=e_h$, $\theta=t_0$, $\nu=t_0-2$,  \eqref{dual-regular} and Theorem \ref{THM:energy-estimate} to obtain
\begin{equation}\label{error-equation-3313}
\begin{split}
&|\zeta_\Psi(e_h)|\\
\lesssim&(h^{\frac{2t_0+\gamma_1+1}{2}}+h^{\frac{2t_0+\gamma_2-1}{2}}+h^{t_0-1}+h^{\frac{2t_0-5-\gamma_1}{2}}+h^{\frac{2t_0-3-\gamma_2}{2}})\|\Psi\|_{t_0+1}\3bare_h\3bar\\
\lesssim&h^{t_0-1}(h^{\frac{\gamma_1+3}{2}}+h^{\frac{\gamma_2+1}{2}}+1+h^{\frac{-\gamma_1-3}{2}}+h^{\frac{-\gamma_2-1}{2}})\|\Psi\|_{4}\3bare_h\3bar\\
\lesssim&h^{t_0-1}(h^{\frac{\gamma_1+3}{2}}+h^{\frac{\gamma_2+1}{2}}+1+h^{\frac{-\gamma_1-3}{2}}+h^{\frac{-\gamma_2-1}{2}})\|e_0\|\3bare_h\3bar\\
\lesssim&h^{t_0-1}(h^{\frac{\gamma_1+3}{2}}+h^{\frac{\gamma_2+1}{2}}+1+h^{\frac{-\gamma_1-3}{2}}+h^{\frac{-\gamma_2-1}{2}})
C_2(h,k,\gamma_1,\gamma_2,s)\|u\|_q\|e_0\|\\
\lesssim&C_4(h,k,\gamma_1,\gamma_2,s,t_0)\|u\|_q\|e_0\|.
\end{split}
\end{equation}

Finally, substituting  \eqref{error-equation-3310}- \eqref{error-equation-3313} into \eqref{error-equation-2} gives rise to the desired result. This completes the proof of the theorem. 
\end{proof}

  \begin{remark}
   For the {\rm gWG} element $V_{k,m,\ell}(T)$  with $m=k-2$, $\ell=k-2$ and $n\leq k-2$,  $\|e_0\|$ has a sub-optimal order of convergence for $k=2$ and an optimal order of convergence for $k\geq3$ when $\gamma_1=-3$ and $\gamma_2=-1$.
  \end{remark}

To establish the error estimates for $e_b$ and $\pmb{e_g}$ in the $L^2$ norm, we introduce
$$
\|e_b\|_{\E_h}=\Big(\sum_{T\in{\cal T}_h}h_T\|e_b\|_{\pa T}^2\Big)^{\frac{1}{2}},~~~~~\|\pmb{e_g}\|_{\E_h}=\Big(\sum_{T\in{\cal T}_h}h_T\|\pmb{e_g}\|_{\pa T}^2\Big)^{\frac{1}{2}}.
$$

\begin{theorem}\label{THM:L2-estimate-edge}
In the assumptions of Theorem \ref{THM:L2-estimate}, the following error estimates hold true:
\begin{equation*}
\begin{split}
&\|e_b\|_{\E_h}\lesssim\|e_0\|+h^\frac{1-\gamma_1}{2}\3bare_h\3bar,\\
&\|\pmb{e_g}\|_{\E_h}\lesssim h^{-1}\|e_0\|+h^\frac{1-\gamma_2}{2} \3bare_h\3bar, 
\end{split}
\end{equation*}
where the estimates for $\3bare_h\3bar$ and $\|e_0\|$ are given by Theorem \ref{THM:energy-estimate} and Theorem \ref{THM:L2-estimate}, respectively. 
\end{theorem}
\begin{proof}
From the triangle inequality and the trace inequality \eqref{inverse inequality}, one arrives at
\begin{equation*}\label{L2-estimate-edge-2}
\begin{split}
 \|e_b\|_{\E_h}
=&\Big(\sum_{T\in{\cal T}_h}h_T\|e_b\|_{\pa T}^2\Big)^{\frac{1}{2}}\\
\lesssim&\Big(\sum_{T\in{\cal T}_h}h_T\|Q_be_0\|_{\pa T}^2+h_T\|e_b-Q_be_0\|_{\pa T}^2\Big)^{\frac{1}{2}}\\
\lesssim&\Big(\sum_{T\in{\cal T}_h}h_Th_T^{-1}\|e_0\|_{T}^2\Big)^{\frac{1}{2}}
    +\Big(\sum_{T\in{\cal T}_h}h_{T}^{1-\gamma_1}\rho_1h_{T}^{\gamma_1}\|e_b-Q_be_0\|_{\pa T}^2\Big)^{\frac{1}{2}}\\
\lesssim&\|e_0\|+h^\frac{1-\gamma_1}{2}\3bare_h\3bar.
\end{split}
\end{equation*}
Similar argument can be straightforward applied to establish the error estimate for $\pmb{e_g}$. This completes the proof of the theorem.
\end{proof}

\section{Numerical experiments}\label{Section:NE}

In this section,  some numerical results are demonstrated to validate the convergence theory established in the previous sections.

For the convenience of notation, the {\rm g}WG element and the generalized discrete weak second order partial derivative are denoted by $P_k(T)|P_m(\pa T)|P_\ell(\pa T)]^d||P_n(T)$ element. The numerical experiments will be conducted on the convex domain $\O_1=(0,1)^2$ and non-convex  L-shaped domain  $\O_2$ with vertices $A_1=(-1,1)$, $A_2=(1,1)$, $A_3=(1,0)$, $A_4=(0,0)$, $A_5=(0,-1)$ and $A_6=(-1,-1)$. The uniform triangular, uniform square, and uniform rectangular partitions are employed. An uniform triangular partition starts with an initial uniform triangulation, and the next level partition is obtained by uniformly partitioning each coarse triangular element into four sub-triangles by connecting the middle points of three edges of each triangular element. To obtain an uniform rectangular partition, we start from an initial $3\times2$ rectangular partition, and the next level partition is obtained by uniformly dividing each rectangular element into four sub-rectangles by connecting the middle points on two parallel edges of each rectangular element. An uniform square partition can be obtained similarly.

The following metrics are employed to measure the errors:
$$
\displaystyle
\begin{array}{lll}
&\mbox{Discrete $H^2$-norm:}&\ \displaystyle|e_0|_{2}=\Big(\sum_{T\in{\cal T}_h}\sum_{i,j=1}^d|\pa_{ij}^2u_0(M_c)-\pa_{ij}^2u(M_c)|^2
                                                                    \times|T|\Big)^{1/2},\\
&\mbox{Discrete $H^{1}$-norm:}&\ \displaystyle|e_0|_{1}=\Big(\sum_{T\in{\cal T}_h}|\nabla u_0(M_c)-\nabla u(M_c)|^2\times|T|\Big)^{1/2},
\end{array}
$$
where $M_c$ represents the center point on each element $T$ and $|T|$ is the area of $T$.

 In all tables, ``Conv''   means the theoretical rate of convergence established in this paper. ``N/A'' means the theoretical rate of convergence has not been developed in this paper.
\subsection{Smooth exact solutions}

\subsubsection{The {\rm gWG} element  with \text{$n>k-2$}}

Table \ref{NE:squa:Case1-2-0-0-1-2} illustrates the numerical performance for the {\rm g}WG scheme  \eqref{WG-scheme} with the  $P_2(T)|P_0(\pa T)|[P_0(\pa T)]^2||P_1(T)$ element on an uniform rectangular partition of the domain $\O_1$. The stabilization parameters are given by $\rho_1=1$, $\rho_2=1$, $\gamma_1=-3$ and $\gamma_2=-1$. The right-hand side function and the boundary conditions are chosen to match the exact solution $u=\cos(x+1)\sin(2y-1)$.  These numerical results are  consistent with the convergence rates for $\3bare_h\3bar$ and $\|e_b\|_{\E_h}$.  $\|\pmb{e_g}\|_{\E_h}$ and  $\|e_{0}\|$ converge at the rates higher than the theoretical rates ${\cal O}(h)$ and ${\cal O}(h^2)$ respectively.  In addition, the convergence rates for  $|e_0|_{2}$ and $|e_0|_{1}$ are demonstrated, for which no theory has been developed in this paper.

\begin{table}[htbp]\centering\scriptsize
\tiny
\begin{center}
\caption{Numerical errors and convergence rates for the $P_2(T)|P_0(\pa T)|[P_0(\pa T)]^2||P_1(T)$ element.}\label{NE:squa:Case1-2-0-0-1-2}
\begin{tabular}{p{0.3cm}p{0.9cm}p{0.2cm}p{0.9cm}p{0.2cm}p{0.9cm}p{0.2cm}p{0.9cm}p{0.5cm}p{0.9cm}p{0.6cm}p{0.9cm}p{0.3cm}}
\hline
  $1/h$&$\3bare_h\3bar$&Rate&$\|e_{0}\|$&Rate&$\|e_b\|_{\E_h}$&Rate&$\|\pmb{e_g}\|_{\E_h}$&Rate&$|e_0|_{2}$&Rate&$|e_0|_{1}$&Rate\\
\hline
                8   &8.73e-02 &0.96    &9.34e-05 &3.63    &5.11e-05 &1.81    &1.92e-03 &1.92    &1.32e-02 &1.43    &2.66e-03 &1.93\\
                16  &4.44e-02 &0.98    &9.00e-06 &3.38    &6.24e-06 &3.04    &4.79e-04 &2.00    &4.67e-03 &1.50    &6.68e-04 &1.99\\
                32  &2.24e-02 &0.99    &1.02e-06 &3.14    &7.09e-07 &3.14    &1.20e-04 &2.00    &1.64e-04 &1.51    &1.67e-04 &2.00\\
                64  &1.13e-02 &0.99    &1.33e-07 &2.94    &1.26e-07 &2.49    &2.99e-05 &2.00    &5.75e-04 &1.51    &4.17e-05 &2.00\\
                128 &5.65e-03 &1.00    &2.07e-08 &2.69    &2.93e-08 &2.11    &7.47e-06 &2.00    &2.03e-04 &1.50    &1.04e-05 &2.00\\
                \text{Conv} &  &1.0     &         &2.0     &         &2.0     &         &1.0     &         &N/A      &         &N/A\\
\hline
\end{tabular}
\end{center}
\end{table}

\subsubsection{The {\rm gWG} element  with \text{$n=k-2$}}

Table \ref{NE:tri:Case1-2-0-1-2} demonstrates the numerical performance of the $P_2(T)|P_m(\pa T)|[P_0(\pa T)]^2||P_0(T)$ elements for $m=0,1,2$ on the uniform triangular partition of $\O_1$. The stabilization parameters are $\rho_1=1$, $\rho_2=1$,  $\gamma_1=-3$ and $\gamma_2=-1$. The exact solution is $u=\sin(x)\sin(y)$. For different values of $m$, we observe the convergence rates for $\3bare_h\3bar$,  $\|e_0\|$ and $\|e_b\|_{\E_h}$ are consistent with the theory, and the convergence rate for $\|\pmb{e_g}\|_{\E_h}$ seems to be higher than the theoretical rate ${\cal O}(h)$.

\begin{table}[htbp]\centering\scriptsize
\tiny
\begin{center}
\caption{Numerical errors and convergence rates for the $P_2(T)|P_m(\pa T)|[P_0(\pa T)]^2||P_0(T)$ element.}\label{NE:tri:Case1-2-0-1-2}
\begin{tabular}{p{0.3cm}p{0.9cm}p{0.2cm}p{0.9cm}p{0.2cm}p{0.9cm}p{0.2cm}p{0.9cm}p{0.2cm}p{1.15cm}p{0.2cm}p{1.15cm}p{0.3cm}}
\hline
  $1/h$&$\3bare_h\3bar$&Rate&$\|e_0\|$&Rate&$\|e_b\|_{\E_h}$&Rate&$\|\pmb{e_g}\|_{\E_h}$&Rate&$|e_0|_{2}$&Rate&$|e_0|_{1}$&Rate\\
\hline
&\text{$m=0$}&&&&&&&&&&&\\
\hline
                8   &1.59e-01 &0.46    &7.00e-03 &1.75    &4.62e-03 &1.72    &4.25e-02 &0.72    &8.84e-02 &0.55    &2.03e-02 &0.74\\
                16  &1.03e-01 &0.63    &2.32e-03 &1.59    &1.60e-03 &1.53    &1.80e-02 &1.24    &4.67e-02 &0.92    &8.59e-03 &1.24\\
                32  &5.92e-02 &0.80    &7.11e-04 &1.71    &5.00e-04 &1.68    &5.98e-03 &1.59    &2.13e-02 &1.14    &2.87e-03 &1.58\\
                64  &3.17e-02 &0.90    &1.95e-04 &1.87    &1.38e-04 &1.86    &1.73e-03 &1.79    &9.25e-03 &1.20    &8.28e-04 &1.79\\
                128 &1.64e-02 &0.95    &5.07e-05 &1.94    &3.58e-05 &1.94    &4.61e-04 &1.91    &4.16e-03 &1.15    &2.21e-04 &1.90\\
                \text{Conv} &  &1.0     &         &2.0     &         &2.0     &         &1.0     &         &N/A      &         &N/A\\
\hline
&\text{$m=1$}&&&&&&&&&&&\\
\hline
                8   &1.43e-01 &0.52    &4.78e-03 &1.66    &5.58e-03 &1.64    &3.29e-02 &0.93    &1.01e-01 &0.48    &1.53e-02 &0.95\\
                16  &8.77e-02 &0.71    &1.55e-03 &1.62    &1.80e-03 &1.64    &1.25e-02 &1.40    &6.10e-02 &0.72    &5.82e-03 &1.39\\
                32  &4.87e-02 &0.85    &4.50e-04 &1.78    &5.20e-04 &1.79    &3.89e-03 &1.69    &3.33e-02 &0.87    &1.82e-03 &1.68\\
                64  &2.56e-02 &0.93    &1.20e-04 &1.91    &1.38e-04 &1.91    &1.08e-03 &1.85    &1.73e-02 &0.94    &5.04e-04 &1.85\\
                128 &1.31e-02 &0.97    &3.08e-05 &1.96    &3.55e-05 &1.96    &2.82e-04 &1.93    &8.82e-03 &0.97    &1.32e-04 &1.93\\
                \text{Conv} &  & 1.0    &         &2.0     &         &2.0     &         &1.0     &         &N/A      &         &N/A\\
\hline
&\text{$m=2$}&&&&&&&&&&&\\
\hline
                8   &1.26e-01 &0.59    &3.35e-03 &1.67    &4.68e-03 &1.63    &2.53e-02 &1.10    &6.84e-02 &0.69    &1.13e-02 &1.12\\
                16  &7.39e-02 &0.77    &1.03e-03 &1.70    &1.45e-03 &1.69    &8.87e-03 &1.51    &3.59e-02 &0.93    &3.97e-03 &1.51\\
                32  &4.01e-02 &0.88    &2.85e-04 &1.85    &4.03e-04 &1.85    &2.62e-03 &1.76    &1.76e-02 &1.03    &1.18e-03 &1.75\\
                64  &2.08e-02 &0.95    &7.43e-05 &1.94    &1.05e-04 &1.94    &7.07e-04 &1.89    &8.60e-03 &1.03    &3.18e-04 &1.89\\
                128 &1.06e-02 &0.98    &1.89e-05 &1.97    &2.68e-05 &1.97    &1.83e-04 &1.95    &4.24e-03 &1.02    &8.24e-05 &1.95\\
                \text{Conv} &  & 1.0    &         &2.0     &         &2.0     &         &1.0     &         &N/A      &         &N/A\\
\hline
\end{tabular}
\end{center}
\end{table}

Table \ref{NE:L-shaped:Case3-3-0-1-1} presents the numerical result  for the $P_3(T)|P_0(\pa T)|[P_1(\pa T)]^2||P_1(T)$ element on the uniform triangular partition of $\O_1$.  We take  $\gamma_1=-3$, $\gamma_2=-1$, and $\rho_i=1$ for $i=1,2$. The exact solution is $u=\cos(x+1)\sin(2y-1)$. The numerical performance is similar to that of the numerical test shown in Table \ref{NE:tri:Case1-2-0-1-2}.

\begin{table}[htbp]\centering\scriptsize
\tiny
\begin{center}
\caption{Numerical errors and convergence rates for the $P_3(T)|P_0(\pa T)|[P_1(\pa T)]^2||P_1(T)$ element.}\label{NE:L-shaped:Case3-3-0-1-1}
\begin{tabular}{p{0.3cm}p{0.9cm}p{0.2cm}p{0.9cm}p{0.2cm}p{0.9cm}p{0.2cm}p{0.9cm}p{0.2cm}p{1.15cm}p{0.2cm}p{1.15cm}p{0.3cm}}
\hline
  $1/h$&$\3bare_h\3bar$&Rate&$\|e_0\|$&Rate&$\|e_b\|_{\E_h}$&Rate&$\|\pmb{e_g}\|_{\E_h}$&Rate&$|e_0|_{2}$&Rate&$|e_0|_{1}$&Rate\\
\hline
                4  &6.33e-01 &0.94    &1.55e-02 &3.77    &4.64e-04 &2.61    &6.12e-02 &1.80    &8.31e-02 &1.54    &8.08e-03 &2.31\\
                8  &3.19e-01 &0.99    &1.12e-03 &3.79    &1.04e-04 &2.16    &1.55e-02 &1.98    &2.71e-02 &1.62    &1.52e-03 &2.41\\
                16 &1.60e-01 &0.99    &1.03e-04 &3.45    &2.60e-05 &2.00    &3.86e-03 &2.01    &8.84e-03 &1.61    &2.93e-04 &2.38\\
                32 &8.03e-02 &1.00    &1.52e-05 &2.76    &6.57e-06 &1.99    &9.60e-04 &2.01    &2.94e-03 &1.59    &5.96e-05 &2.30\\
                64 &4.02e-02 &1.00    &3.20e-06 &2.25    &1.65e-06 &1.99    &2.39e-04 &2.00    &9.94e-04 &1.56    &1.29e-05 &2.21\\
                \text{Conv} &  &1.0     &         &2.0     &         &2.0     &         &1.0     &         &N/A      &         &N/A\\
\hline
\end{tabular}
\end{center}
\end{table}

Table \ref{NE:TRI:Case1-2-0-0-0} shows the computational result for the $P_2(T)|P_0(\pa T)|[P_0(\pa T)]^2||P_0(T)$ element on the uniform triangular partition of $\O_1$.   We take $\gamma_1=-3$ and $\gamma_2=-1$. The exact solution is  $u=\cos(x)\sin(y)$. For different values of $\rho_i$ for $i=1,2$, we observe that the convergence rates for $\3bare_h\3bar$, $\|e_0\|$ and $\|e_b\|_{\E_h}$ are consistent with the theory, and the convergence rate  for $\|\pmb{e_g}\|_{\E_h}$   outperforms the theory.

\begin{table}[htbp]\centering\scriptsize
\tiny
\begin{center}
\caption{Numerical errors and convergence rates for the $P_2(T)|P_0(\pa T)|[P_0(\pa T)]^2||P_0(T)$ element.}\label{NE:TRI:Case1-2-0-0-0}
\begin{tabular}{p{0.3cm}p{0.9cm}p{0.2cm}p{0.9cm}p{0.2cm}p{0.9cm}p{0.2cm}p{0.9cm}p{0.4cm}p{1.00cm}p{0.2cm}p{1.0cm}p{0.3cm}}
\hline
  $1/h$&$\3bare_h\3bar$&Rate&$\|e_0\|$&Rate&$\|e_b\|_{\E_h}$&Rate&$\|\pmb{e_g}\|_{\E_h}$&Rate&$|e_0|_{2}$&Rate&$|e_0|_{1}$&Rate\\
\hline
&\text{$\rho_1=1$,~~$\rho_2=1$}&&&&&&&&&&&\\
\hline
                8   &2.35e-01 &0.34    &8.64e-03 &2.05    &5.62e-03 &2.07    &7.66e-02 &0.45    &1.79e-01 &0.41    &3.71e-02 &0.43\\
                16  &1.63e-01 &0.53    &2.28e-03 &1.92    &1.51e-03 &1.90    &3.91e-02 &0.97    &9.86e-02 &0.86    &1.91e-02 &0.96\\
                32  &9.63e-02 &0.76    &6.50e-04 &1.81    &4.50e-04 &1.75    &1.37e-02 &1.51    &4.18e-02 &1.24    &6.70e-03 &1.51\\
                64  &5.17e-02 &0.90    &1.82e-04 &1.83    &1.28e-04 &1.81    &3.94e-03 &1.80    &1.70e-02 &1.30    &1.92e-03 &1.80\\
                128 &2.67e-02 &0.96    &4.87e-05 &1.91    &3.44e-05 &1.90    &1.04e-03 &1.92    &7.37e-03 &1.20    &5.09e-04 &1.92\\
                \text{Conv} &  &1.0     &         &2.0     &         &2.0     &         &1.0     &         &N/A      &         &N/A\\
\hline
&\text{$\rho_1=10$,~~$\rho_2=10$}&&&&&&&&&&&\\
\hline
                8   &1.32e-01 &0.76    &1.08e-03 &2.00    &5.50e-04 &1.81    &1.87e-02 &1.22    &5.21e-02 &1.07    &9.28e-03 &1.22\\
                16  &7.23e-02 &0.87    &2.64e-04 &2.03    &1.66e-04 &1.73    &5.95e-03 &1.65    &2.21e-02 &1.24    &2.96e-03 &1.65\\
                32  &3.76e-02 &0.94    &6.88e-04 &1.94    &4.72e-05 &1.82    &1.64e-03 &1.86    &9.72e-03 &1.19    &8.18e-04 &1.86\\
                64  &1.92e-02 &0.97    &1.80e-05 &1.93    &1.27e-05 &1.90    &4.28e-04 &1.94    &4.56e-03 &1.09    &2.14e-04 &1.94\\
                128 &9.66e-03 &0.99    &4.65e-06 &1.96    &3.28e-06 &1.95    &1.09e-04 &1.97    &2.22e-03 &1.04    &5.44e-05 &1.97\\
                \text{Conv} &  &1.0     &         &2.0     &         &2.0     &          &1.0    &          &N/A     &         &N/A\\
\hline
&\text{$\rho_1=100$,~~$\rho_2=1$}&&&&&&&&&&&\\
\hline
                8   &9.45e-02 &0.81    &7.68e-04 &2.27    &4.11e-04 &1.88    &1.25e-02 &1.60    &8.80e-02 &0.69    &3.45e-03 &1.57\\
                16  &5.00e-02 &0.92    &1.69e-04 &2.18    &1.09e-04 &1.92    &3.49e-03 &1.84    &4.82e-02 &0.87    &9.75e-04 &1.82\\
                32  &2.56e-02 &0.97    &4.05e-05 &2.06    &2.79e-05 &1.96    &9.13e-04 &1.93    &2.51e-02 &0.94    &2.56e-04 &1.93\\
                64  &1.30e-02 &0.98    &1.01e-05 &2.01    &7.06e-06 &1.98    &2.33e-04 &1.97    &1.28e-02 &0.97    &6.54e-05 &1.97\\
                128 &6.51e-03 &0.99    &2.52e-06 &2.00    &1.78e-06 &1.99    &5.87e-05 &1.99    &6.44e-03 &0.99    &1.65e-05 &1.99\\
                \text{Conv} &  &1.0     &         &2.0     &         &2.0     &         &1.0     &         &N/A      &         &N/A\\
\hline
\end{tabular}
\end{center}
\end{table}

\subsubsection{The {\rm gWG} element with \text{$n<k-2$}}

In Table \ref{NE:tri:Case3-3-2-0-0-1}, the numerical errors and convergence rates are shown for the $P_3(T)|P_m(\pa T)|[P_\ell(\pa T)]^2||P_0(T)$ element on the uniform triangular partition of $\O_1$. We set $\gamma_1=-3$, $\gamma_2=-1$, $\rho_1=1$ and  $\rho_2=1$.
The exact solution is $u=\cos(x+1)\sin(2y-1)$. We can see from Table \ref{NE:tri:Case3-3-2-0-0-1} that the convergence rates for $\3bare_h\3bar$, $\|e_0\|$ and $\|e_b\|_{\E_h}$ are a bit lower than the theoretical rates respectively. However, the convergence rate for  $\|\pmb{e_g}\|_{\E_h}$ is higher than the theoretical order ${\cal O}(h)$.

\begin{table}[htbp]\centering\scriptsize
\tiny
\begin{center}
\caption{Numerical errors and convergence rates for the $P_3(T)|P_m(\pa T)|[P_\ell(\pa T)]^2||P_0(T)$ element.}\label{NE:tri:Case3-3-2-0-0-1}
\begin{tabular}{p{0.3cm}p{0.9cm}p{0.2cm}p{0.9cm}p{0.2cm}p{0.9cm}p{0.2cm}p{0.9cm}p{0.2cm}p{1.15cm}p{0.2cm}p{1.15cm}p{0.3cm}}
\hline
  $1/h$&$\3bare_h\3bar$&Rate&$\|e_0\|$&Rate&$\|e_b\|_{\E_h}$&Rate&$\|\pmb{e_g}\|_{\E_h}$&Rate&$|e_0|_{2}$&Rate&$|e_0|_{1}$&Rate\\
\hline
&\text{$m=0$,~~$\ell=1$}&&&&&&&&&&&\\
\hline
                4  &6.06e-01 &0.73    &3.85e-02 &2.23    &1.85e-02 &1.82    &1.38e-01 &0.84    &2.23e-01 &0.52    &5.30e-02 &0.71\\
                8  &4.06e-01 &0.58    &1.39e-02 &1.47    &7.18e-03 &1.37    &1.05e-01 &0.40    &1.99e-01 &0.17    &3.81e-02 &0.48\\
                16 &2.86e-01 &0.51    &6.99e-03 &1.00    &3.75e-03 &0.94    &6.52e-02 &0.69    &1.38e-01 &0.53    &2.31e-02 &0.72\\
                32 &1.81e-01 &0.66    &2.75e-03 &1.35    &1.50e-03 &1.32    &2.73e-02 &1.25    &6.54e-02 &1.07    &9.65e-03 &1.26\\
                64 &1.02e-01 &0.83    &8.42e-04 &1.71    &4.61e-04 &1.70    &8.63e-03 &1.66    &2.48e-02 &1.40    &3.04e-03 &1.67\\
                \text{Conv} &  &1.0     &         &2.0     &         &2.0     &         &1.0     &         &N/A     &         &N/A\\
\hline
&\text{$m=2$,~~$\ell=1$}&&&&&&&&&&&\\
\hline
                4  &5.71e-01 &0.71    &1.98e-02 &1.82    &2.06e-02 &1.47    &1.21e-01 &0.99    &1.96e-01 &0.60    &4.45e-02 &0.78\\
                8  &3.61e-01 &0.66    &8.54e-03 &1.21    &9.28e-03 &1.15    &7.40e-02 &0.71    &1.44e-01 &0.45    &2.65e-02 &0.75\\
                16 &2.23e-01 &0.70    &3.47e-03 &1.30    &3.79e-03 &1.29    &3.38e-02 &1.13    &7.54e-02 &0.94    &1.19e-02 &1.15\\
                32 &1.25e-01 &0.83    &1.10e-03 &1.66    &1.20e-03 &1.66    &1.13e-02 &1.59    &2.99e-02 &1.34    &3.95e-03 &1.59\\
                64 &6.65e-02 &0.92    &3.05e-04 &1.85    &3.34e-04 &1.85    &3.19e-03 &1.82    &1.08e-02 &1.47    &1.11e-03 &1.82\\
                \text{Conv} &  &1.0     &         &2.0     &         &2.0     &          &1.0    &         &N/A     &         &N/A\\
\hline
&\text{$m=2$,~~$\ell=0$}&&&&&&&&&&&\\
\hline
                4  &4.75e-01 &0.53    &1.85e-02 &1.78    &1.87e-02 &1.44    &1.00e-01 &0.46    &3.12e-01 &0.00    &4.49e-02 &0.46\\
                8  &3.37e-01 &0.49    &8.67e-03 &1.09    &9.36e-03 &1.00    &6.26e-02 &0.68    &2.28e-01 &0.45    &2.81e-02 &0.67\\
                16 &2.20e-01 &0.62    &3.64e-03 &1.25    &3.98e-03 &1.24    &2.82e-02 &1.15    &1.29e-01 &0.81    &1.27e-02 &1.14\\
                32 &1.26e-01 &0.80    &1.16e-03 &1.65    &1.27e-03 &1.64    &9.33e-03 &1.60    &6.41e-02 &1.01    &4.22e-03 &1.59\\
                64 &6.75e-02 &0.91    &3.23e-04 &1.85    &3.54e-04 &1.85    &2.63e-03 &1.83    &3.11e-02 &1.04    &1.19e-03 &1.83\\
                \text{Conv} &  &1.0     &         &2.0     &         &2.0     &         &1.0     &         &N/A     &         &N/A\\
\hline
\end{tabular}
\end{center}
\end{table}

\subsection{Low regularity solutions}
We shall numerically demonstrate the performance of the {\rm gWG} method for low regularity solutions. The stabilization parameters are $\gamma_1=-3$, $\gamma_2=-1$ and $\rho_i=1$ for $i=1,2$.

Table \ref{NE:L-shaped:Case1-2-1-0-0} illustrates the numerical performance of $P_2(T)|P_0(\pa T)|[P_0(\pa T)]^2||P_n(T)$ element on the uniform square partition of $\O_1$. The exact solution is given by $u=(x^2+y^2)^{0.8}$. It is obvious that the exact solution $u$ is in $H^{2.6-\varepsilon}$ for arbitrary small $\varepsilon>0$. One can easily observe  from Table \ref{NE:L-shaped:Case1-2-1-0-0} that: 1)  $\3bare_h\3bar$ converges in an order ${\cal O}(h^{0.6})$; 2) The convergence rates for $\|e_0\|$ and $\|e_b\|_{\E_h}$ seem to be around ${\cal O}(h^{2.4})$ for the case of $n=0$ and ${\cal O}(h^{2.6})$ for the case of $n=1$, respectively; 3) The convergence rates for  $\|\pmb{e_g}\|_{\E_h}$, $|e_0|_{2}$ and $|e_0|_{1}$ seem to be ${\cal O}(h^{1.6})$, ${\cal O}(h^{0.6})$ and ${\cal O}(h^{1.6})$, respectively. Note that the convergence theory developed in the previous sections is not available to the case with the solution in low regularity.

\begin{table}[htbp]\centering\scriptsize
\tiny
\begin{center}
\caption{Numerical errors and convergence rates for the $P_2(T)|P_0(\pa T)|[P_0(\pa T)]^2||P_n(T)$ element.}\label{NE:L-shaped:Case1-2-1-0-0}
\begin{tabular}{p{0.3cm}p{0.9cm}p{0.2cm}p{0.9cm}p{0.2cm}p{0.9cm}p{0.2cm}p{0.9cm}p{0.2cm}p{1.15cm}p{0.2cm}p{1.15cm}p{0.3cm}}
\hline
  $1/h$&$\3bare_h\3bar$&Rate&$\|e_0\|$&Rate&$\|e_b\|_{\E_h}$&Rate&$\|\pmb{e_g}\|_{\E_h}$&Rate&$|e_0|_{2}$&Rate&$|e_0|_{1}$&Rate\\
\hline
&\text{$n=0$}&&&&&&&&&&&\\
\hline
                8   &1.42e-00 &0.58    &1.40e-02 &2.47    &2.03e-02 &2.35    &7.06e-02 &1.00    &6.12e-01 &0.59    &5.60e-02 &1.25\\
                16  &9.43e-01 &0.59    &2.52e-03 &2.47    &3.90e-03 &2.38    &2.83e-02 &1.32    &4.02e-01 &0.61    &2.12e-02 &1.40\\
                32  &6.25e-01 &0.59    &4.48e-04 &2.49    &7.25e-04 &2.43    &1.00e-02 &1.50    &2.65e-01 &0.60    &7.38e-03 &1.52\\
                64  &4.13e-01 &0.60    &7.98e-05 &2.49    &1.34e-04 &2.44    &3.41e-03 &1.55    &1.75e-01 &0.60    &2.50e-03 &1.56\\
                128 &2.73e-01 &0.60    &1.48e-05 &2.44    &2.57e-05 &2.38    &1.14e-03 &1.57    &1.15e-01 &0.60    &8.35e-04 &1.58\\
                \text{Conv}   &         &N/A   &         &N/A     &         &N/A     &         &N/A     &         &N/A     &         &N/A\\
\hline
&\text{$n=1$}&&&&&&&&&&&\\
\hline
                8   &1.20e-00 &0.60    &4.76e-03 &2.60    &6.36e-04 &1.65    &4.16e-03 &1.04    &5.18e-01 &0.61    &3.48e-02 &1.59\\
                16  &7.90e-01 &0.60    &1.28e-03 &2.60    &1.34e-04 &2.25    &1.39e-03 &1.58    &3.42e-01 &0.60    &1.15e-02 &1.60\\
                32  &5.21e-01 &0.60    &2.11e-04 &2.60    &2.25e-05 &2.57    &4.61e-04 &1.59    &2.25e-01 &0.60    &3.79e-03 &1.60\\
                64  &3.44e-01 &0.60    &3.48e-05 &2.60    &3.72e-06 &2.60    &1.53e-04 &1.60    &1.49e-01 &0.60    &1.25e-03 &1.60\\
                128 &2.27e-01 &0.60    &5.75e-06 &2.60    &6.13e-07 &2.60    &5.04e-05 &1.60    &9.81e-02 &0.60    &4.12e-04 &1.60\\
                \text{Conv} &  &N/A   &         &N/A     &         &N/A     &         &N/A     &         &N/A     &         &N/A \\
\hline
\end{tabular}
\end{center}
\end{table}

Table \ref{NE:L-shaped:CaseLS-2-0-0-0} shows the numerical results for the $P_2(T)|P_0(\pa T)|[P_0(\pa T)]^2||P_0(T)$ element. The right-hand side function and boundary conditions are chosen such that the exact solution is $u=r^{3/2}\sin (\frac{\theta}{2})$ in which $r=\sqrt{x^2+y^2}$ and $\theta=\arctan(y/x)$. It is obvious $u\in H^{2.5-\varepsilon}$ for arbitrary small $\varepsilon>0$. We observe from Table \ref{NE:L-shaped:CaseLS-2-0-0-0}  that: 1) the convergence rates for the errors   on the uniform square partition of $\O_1$ are slightly  higher than those on the uniform triangular partition of $\O_1$ respectively; 2) the convergence rates for  $\3bare_h\3bar$ and $|e_0|_2$ on the uniform triangular partition of the convex domain $\O_1$ are almost the same with those  on the uniform triangular partition of the non-convex domain $\O_2$; 3) the convergence rates for  $\|e_0\|$, $\|e_b\|_{\E_h}$, $\|\pmb{e_g}\|_{\E_h}$ and $|e_0|_1$ on the uniform triangular partition of the convex domain $\O_1$ are a bit higher than those  on the uniform triangular partition of the non-convex domain $\O_2$. Again,   the convergence theory developed in the previous sections is not available to the case with the solution in low regularity. Readers are welcome to draw their own conclusions. 

\begin{table}[htbp]\centering\scriptsize
\tiny
\begin{center}
\caption{Numerical errors and convergence rates for the $P_2(T)|P_0(\pa T)|[P_0(\pa T)]^2||P_0(T)$ element.}\label{NE:L-shaped:CaseLS-2-0-0-0}
\begin{tabular}{p{0.3cm}p{0.9cm}p{0.2cm}p{0.9cm}p{0.2cm}p{0.9cm}p{0.2cm}p{0.9cm}p{0.2cm}p{1.15cm}p{0.2cm}p{1.15cm}p{0.3cm}}
\hline
  $1/h$&$\3bare_h\3bar$&Rate&$\|e_0\|$&Rate&$\|e_b\|_{\E_h}$&Rate&$\|\pmb{e_g}\|_{\E_h}$&Rate&$|e_0|_{2}$&Rate&$|e_0|_{1}$&Rate\\
\hline
&\text{$\O_1$,~uniform~square~partition}&&&&&&&&&&&\\
\hline
                   8   &9.50e-02 &0.47    &5.99e-04 &1.47    &1.20e-03 &1.46    &9.11e-03 &1.34    &4.57e-02 &0.71    &3.39e-03 &1.13\\
                   16  &6.34e-02 &0.58    &1.73e-04 &1.79    &3.46e-04 &1.79    &3.14e-03 &1.54    &2.83e-02 &0.69    &1.22e-03 &1.47\\
                   32  &4.12e-02 &0.62    &4.49e-05 &1.95    &8.97e-05 &1.95    &1.03e-03 &1.61    &1.77e-02 &0.67    &4.08e-04 &1.58\\
                   64  &2.64e-02 &0.64    &1.12e-05 &2.00    &2.24e-05 &2.00    &3.32e-04 &1.63    &1.12e-02 &0.67    &1.32e-04 &1.63\\
                   128 &1.68e-02 &0.65    &2.76e-06 &2.02    &5.53e-06 &2.02    &1.06e-04 &1.65    &7.03e-03 &0.67    &4.23e-05 &1.64 \\
                \text{Conv}   &        &N/A   &         &N/A     &         &N/A     &         &N/A     &         &N/A     &         &N/A\\
\hline
&\text{$\O_1$,~uniform~triangular~partition}&&&&&&&&&&&\\
\hline
                   8   &3.77e-01 &0.11    &5.78e-03 &1.05    &2.86e-03 &0.49    &1.03e-01 &0.39    &3.03e-01 &0.17    &4.97e-02 &0.38\\
                   16  &3.13e-01 &0.27    &2.52e-03 &1.20    &1.58e-03 &0.85    &5.77e-02 &0.83    &2.29e-01 &0.40    &2.81e-02 &0.82\\
                   32  &2.37e-01 &0.40    &9.51e-04 &1.41    &6.50e-04 &1.28    &2.46e-02 &1.23    &1.63e-01 &0.49    &1.20e-02 &1.23\\
                   64  &1.73e-01 &0.46    &3.08e-04 &1.62    &2.16e-04 &1.59    &9.36e-03 &1.40    &1.15e-01 &0.50    &4.56e-03 &1.40\\
                   128 &1.24e-01 &0.48    &8.98e-05 &1.78    &6.32e-05 &1.77    &3.41e-03 &1.46    &8.16e-02 &0.50    &1.66e-03 &1.46\\
                \text{Conv} &   &N/A   &         &N/A     &         &N/A     &         &N/A     &         &N/A     &         &N/A\\
\hline
&\text{$\O_2$,~uniform~triangular~partition}&&&&&&&&&&&\\
\hline
                   4  &3.44e-01 &0.43    &1.17e-02 &2.18    &4.24e-03 &0.83    &6.78e-02 &1.19    &3.21e-01 &0.41    &2.98e-02 &1.16\\
                   8  &2.48e-01 &0.47    &3.87e-03 &1.59    &2.44e-03 &0.80    &2.65e-02 &1.36    &2.32e-01 &0.46    &1.18e-02 &1.34\\
                   16 &1.77e-01 &0.49    &1.81e-03 &1.09    &1.26e-03 &0.95    &1.06e-02 &1.32    &1.66e-01 &0.48    &4.82e-03 &1.29\\
                   32 &1.26e-01 &0.49    &8.39e-04 &1.11    &5.92e-04 &1.09    &4.40e-03 &1.26    &1.18e-01 &0.49    &2.05e-03 &1.23\\
                   64 &8.91e-02 &0.50    &3.81e-04 &1.14    &2.70e-04 &1.14    &1.87e-03 &1.24    &8.39e-02 &0.50    &8.85e-04 &1.21\\
                \text{Conv}   &        &N/A   &         &N/A     &         &N/A     &         &N/A     &         &N/A     &         &N/A\\
\hline
\end{tabular}
\end{center}
\end{table}

\bigskip
\bigskip

\vfill\eject

\end{document}